\documentclass[12pt]{amsart}
\usepackage{amssymb,amscd}
\usepackage[arrow,matrix,curve]{xy}
\usepackage{fullpage}
\usepackage{graphicx}

\newcommand\CO{\mathbb{C}}
\newcommand\RE{\mathbb{R}}
\newcommand\RR{\mathbb{R}}
\newcommand\R{\mathbb{R}}

\newcommand\ST{\mathbb{S}}
\newcommand\C{\mathbb{C}}
\newcommand\LL{\mathbb{L}}

\usepackage{color}
\definecolor{darkgreen}{cmyk}{1,0,1,.2}
\definecolor{m}{rgb}{1,0.1,1}
\definecolor{green}{cmyk}{1,0,1,0}

\definecolor{test}{rgb}{1,0,0}
\definecolor{cmyk}{cmyk}{0,1,1,0}

\newcommand\Index{\operatorname{Index}}
\newcommand\Tr{\operatorname{Tr}}

\newcommand\Rank{\operatorname{Rank}}

\newcommand\Ker{\operatorname{Ker}}
\newcommand\Hom{\operatorname{Hom}}

\newcommand\GL{\operatorname{GL}}
\newcommand\Sign{\operatorname{Sign}}

\newcommand\tr{\operatorname{tr}}

\newcommand\im{\operatorname{im}}

\newcommand\E{\mathcal E}
\newcommand\cH{\mathcal H}
\newcommand\ep{\epsilon}

\newcommand\maD{\mathcal D}
\newcommand\maT{\mathcal T}
\newcommand\maS{\mathcal S}

\newcommand\maA{\mathcal A}
\newcommand\maB{\mathcal B}
\newcommand\maJ{\mathcal J}

\newcommand\maM{\mathcal M}
\newcommand\tM{\tilde M}

\newcommand\odd{^{\mathrm{odd}}}
\newcommand\even{^{\mathrm{even}}}

\def\GL{\operatorname{GL}}

\theoremstyle{plain}
\newtheorem{theorem}{Theorem}[section]
\newtheorem{lemma}[theorem]{Lemma}
\newtheorem{proposition}[theorem]{Proposition}
\newtheorem{corollary}[theorem]{Corollary}

\theoremstyle{definition}
\newtheorem{definition}[theorem]{Definition}
\newtheorem{definition*}{Definition}

\theoremstyle{remark}
\newtheorem{remark}[theorem]{Remark}

\newtheorem{remarks*}{Remarks}

\begin{document}

\title[Index type invariants for twisted signature complexes]
{Index type invariants for twisted signature complexes\\ and homotopy invariance}
\author{Moulay Tahar Benameur}
\address{Laboratoire et D\'epartement de Math\'ematiques,
  UMR 7122,
  Universit\'e de Metz et CNRS,
  B\^at. A, Ile du Saulcy,
  F-57045 Metz Cedex 1, 
  France}
\email{benameur@math.univ-metz.fr}

\author{Varghese Mathai}
\address{Department of Mathematics, University of Adelaide,
Adelaide 5005, Australia}
\email{mathai.varghese@adelaide.edu.au}

\begin{abstract}
For a closed, oriented, odd dimensional manifold $X$, 
we define the rho invariant $\rho(X,\E,H)$ for the twisted odd signature operator
valued in a flat hermitian vector bundle $\E$,
where $H = \sum i^{j+1} H_{2j+1} $ is an odd-degree closed differential form on $X$ 
and $H_{2j+1}$ is a real-valued differential form of degree ${2j+1}$.
We show that  $\rho(X,\E,H)$ is independent of the
choice of metrics on $X$ and $\E$ and of the representative $H$ in the
cohomology class $[H]$.
We establish some basic functorial properties of the twisted rho invariant.
We express the twisted eta invariant in terms of spectral flow and the usual eta invariant.
In particular, we get a simple expression for it on closed oriented 
3-dimensional manifolds with a degree three flux form. 
A core technique used is our analogue of the Atiyah-Patodi-Singer theorem,
which we establish for the twisted signature operator on a compact, oriented manifold with boundary.
The homotopy invariance of the rho invariant $\rho(X,\E,H)$ is more delicate to establish, 
and is settled under further hypotheses on the fundamental group of $X$.
\end{abstract}

\keywords{twisted rho invariant, twisted eta invariant, twisted de Rham cohomology, 
twisted signature complex, twisted Atiyah-Patodi-Singer theorem, manifolds with boundary, spectral flow, 
homotopy invariance}

\subjclass[2010]{Primary 58J52; Secondary 57Q10, 58J40, 81T30.}

\maketitle

\section*{Introduction}

Atiyah, Patodi and Singer wrote a remarkable series of three papers \cite{APS1, APS2, APS3} 
that were published in the mid seventies, on non-local elliptic boundary value problems,
which have been intensely studied ever since, both in the mathematics and physics literature.
They applied their theory in particular to the important case of the signature 
operator on an oriented, compact manifold with boundary, where they identified 
the boundary operator which is now known as the {\em odd signature operator},
which is self-adjoint and elliptic, having spectrum in the real numbers. For this 
(and other elliptic self-adjoint operators), they introduced the eta invariant
which measures the spectral asymmetry of the operator and is a 
spectral invariant. Coupling with flat bundles, they introduced the closely related rho invariant,
which has the striking property that it is independent of the choice of Riemannian
metric needed in its definition. In this paper we generalize the construction of 
Atiyah-Patodi-Singer to the twisted signature complex with an odd-degree differential 
form as flux and with coefficients in a flat vector bundle.
Recall that the twisted de Rham complex was first defined for $3$-form fluxes by Rohm and
Witten in the appendix of \cite{RW} and has played an important role in string
theory \cite{BCMMS,AS}, for the Ramond-Ramond fields (and their charges) in
type II string theories lie in the twisted cohomology of spacetime.
$T$-duality of the type II string theories on compactified spacetime gives
rise to a duality isomorphism of twisted cohomology groups \cite{BEM}.
The twisted de Rham differential also appears in supergravity (cf. \cite{VanN,GHR}) and superstring theory (\cite{St86})
via Riemannian connections $\nabla$ with totally skew-symmetric, (de Rham) closed torsion tensor $H$,
that is, $\nabla_X Y = \nabla^g_X Y + \frac{1}{2} H (X, Y, -)$, where $\nabla^g$ denotes the 
Levi-Civita connection of the Riemannian metric $g$. 
Such Riemannian connections $\nabla$ have the same geodesics as the Levi-Civita connection $\nabla^g$.
Then the analogue of the de Rham 
differential, $\sum {e^i} \wedge \nabla_{e^i} = d + H\wedge = d_H,$ where $\{e^i\}$ is a local orthonormal 
basis,  is exactly the twisted de Rham differential.
The associated signature operator, $d_H + {d_H}^\dagger$ is precisely the twisted signature operator,
where ${}^\dagger$ denotes the adjoint.

Let $Y$ be a $2m$-dimensional compact, oriented Riemannian manifold without boundary, $\E$
a flat hermitian vector bundle over $Y$. If one uses the standard signature involution 
$\tau$ on $Y$, 
and consider signature operator $B^\E_H=\nabla^{\E, H} + {\nabla^{\E, H}}^\dagger,$ 
where $\nabla^{\E, H}=\nabla^\E+H\wedge\,\cdot\,$, 
$\nabla^\E$ is the canonical flat hermitian connection on $\E$, 
$ {\nabla^{\E, H}}^\dagger$ denotes its adjoint of $ \nabla^{\E, H}$, 
then the first surprising fact is that 
$$
B^\E_H \tau = - \tau B^\E_H
$$
if and only if 
\begin{equation}\label{eqn:fluxform}
H = \sum i^{j+1} H_{2j+1} 
\end{equation}
 is an 
odd-degree closed differential form on $Y$ and $H_{2j+1}$ is a real-valued differential form 
of degree ${2j+1}$. It is only in this case that one gets a {{generalization of the usual}} signature operator, in contrast to the 
case of the twisted de Rham complex, cf . \cite{BCMMS,AS,MW,MW2,MW3}. 
This is also the reason why we consider complex valued differential forms and  cohomology, 
in contrast to the previously mentioned literature. A topological reason for this phenomenon
is that Poincar\'e duality in our context asserts that there is a nondegenerate sesquilinear pairing
$$
H^{\bullet}(Y, \E, H) \otimes H^{\bullet}(Y, \E, -\overline H) \to \CO.
$$
It determines a quadratic form on $H^{\bullet}(Y, \E, H)$ 
when $-\overline H = H$. This implies that $H$ has to be 
purely imaginary. Using a particular scale invariance of twisted cohomology 
as in Lemma \ref{scaling}, we deduce that this is equivalent to 
$H$ being of the form given in \eqref{eqn:fluxform}.

{{When the compact manifold $Y$ has non empty boundary and }}assuming that the 
Riemannian metric is of product type near the boundary {{and that $H$ satisfies the absolute boundary condition}}, we explicitly identify 
the twisted signature operator near the boundary
to be 
$B^\E_H = \sigma\left(\frac{\partial}{\partial r} + D^\E_H \right),$ where 
$r$ is the coordinate in the normal direction, $\sigma$ is a bundle isomorphism and 
finally,  {{$D^\E_H$ is }} the self-adjoint elliptic operator  given on $\Omega^{2h}(\partial Y, \E)$ by
${{D^\E_H = i^m (-1)^{h} (\nabla^{\E, H} \star -  \star \nabla^{\E, -\overline H})}}$. 
There is a similar expression for the operator 
on odd degree forms on the boundary. $D^\E_H$ is an elliptic self-adjoint operator, and 
by \cite{APS3}, the non-local boundary condition given by $P^+(s\big|_{\partial Y})=0$
where $P^+$ denotes the orthogonal projection onto the eigenspaces with positive eigenvalues,
makes the pair $(B^\E_H; P^+)$
into an elliptic boundary value problem. By the Atiyah-Patodi-Singer index theorem,
$$
\Index(B^\E_H; P^+) =  {{\Rank (\E) }}\int_Y \alpha_0(x) dx - \left(\dim(\ker(D^\E_H))
+\eta(D^\E_H)\right),
$$
where $\eta(D^\E_H)$ denotes the eta invariant of the operator $D^\E_H$. Upon identifying the sum
$\Index(\nabla^{\E, H} + \nabla^{\E, H}; P^+) + \dim(\ker(D^\E_H))$ with the signature ${\rm Sign}(Y, \E, H)$ of a natural 
quadratic form on the image 
of the twisted relative cohomology $H^\bullet(Y, \partial Y, \E, H)$ inside the twisted absolute 
cohomology $H^\bullet(Y,  \E, H)$ (Theorem \ref{thm:sig}), one obtains (Corollary \ref{cor:sig})
\begin{equation}\label{sign-thm}
{\rm Sign}(Y, \E, H)= {{\Rank (\E) }}\int_Y \alpha_0(x) dx - \eta(D^\E_H).
\end{equation}
This is the main tool used to prove our results about the twisted rho invariant.

Let now $X$ be a closed, oriented, {{$2m - 1$}} dimensional Riemannian manifold and 
$H = \sum i^{j+1} H_{2j+1} $ an 
odd-degree closed differential form on $X$ where $H_{2j+1}$ is a real-valued differential form 
of degree ${2j+1}$. Denote by $\E$ a hermitian flat vector bundle over $X$ with the 
canonical flat connection $\nabla^\E$.
Consider the twisted odd signature operator acting on even degree forms by the formula
$$
{{D^\E_H = i^m (-1)^{ h} ( \nabla^{\E, H}\star_\E -  \star_\E \nabla^{\E, -{\overline{H}}})}} \text{ on }
\Omega^{2h}(X, \E).
$$
Then $D^{{\E}}_H$ is a self-adjoint
elliptic operator which has a well defined  eta invariant  $\eta(D^\E_H )$. The twisted rho invariant 
$\rho(X, \E, H)$ is defined to be the difference 
$$
\rho(X, \E, H) = \eta(D^\E_H ) - \Rank (\E) \, \eta(D_H ),
$$
where 
$D_H$ is the same signature operator corresponding to the trivial line bundle.
Although the eta invariant  $\eta(D^\E_H )$  is only a spectral invariant, the striking thing is that 
the twisted rho invariant $\rho(X,\E,H)$ is independent of the
choice of the Riemannian metric on $X$ and of the Hermitian metric on $\E$, and is therefore a differential invariant. 
This is analogous to the well known stability properties of the classical  APS rho invariant, see for instance \cite{BP, BR}.
The twisted rho invariant $\rho(X,\E,H)$ is also invariant if $H$ is deformed within its cohomology class.
We also establish some basic functorial properties of this twisted rho invariant.
We compute the twisted eta invariant using the important technique of spectral flow, 
and obtain in particular a simple formula for it for closed oriented 3-dimensional manifolds with a 
degree three flux form.

The homotopy invariance of the rho invariant $\rho(X,\E,H)$ is more delicate to establish, 
and is settled in section \ref{sect:homotopy} under the assumption that the maximal Baum-Connes conjecture holds
for the fundamental group of $X$. 
Discrete groups with the Haagerup property \cite{Haagerup}, also called a-T-menable groups, satisfy  the maximal Baum-Connes conjecture. These are groups that admit an isometric action on an affine Hilbert space which is proper, and examples include: all amenable groups, Coxeter groups, groups acting properly on trees, 
discrete subgroups of $SO(n,1)$ and of $SU(n,1)$. Discrete 
groups satisfying Kazhdan's ``property T'' do {\em not} satisfy the maximal Baum-Connes conjecture.
We adapt the method of analytic surgery 
introduced by Higson and Roe in \cite{HigsonRoeRho} 
applied to the usual rho invariant, which in turn both adds structure to and simplifies Keswani's original proof \cite{Keswani}.

The twisted analogue of analytic torsion was studied in \cite{MW, MW2} and 
is another source of inspiration for this paper. We mention that twisted Dirac operators,
known as {\em cubic Dirac operators} have been studied in representation theory
of Lie groups on homogeneous spaces \cite{Kostant, Slebarski}. It  also appears in the study of Dirac operators on 
loop groups and their representation theory \cite{Landweber}.   In recent work, we study the {{eta}} and rho invariants for the twisted Dirac operator and the relation to positive scalar curvature, \cite{BM, BM2}.

\medskip

\noindent {\bf Acknowledgments}
M.B. ~acknowledges support from the CNRS institute INSMI especially from the PICS-CNRS ``Progress in Geometric Analysis and Applications''.
V.M.~acknowledges support from the Australian Research Council.
Part of this work was done while the second author was visiting the ``Laboratoire de Math\'ematiques et Applications'' at Metz. This work also benefited from discussions with many colleagues during the NCG workshop at Oberwolfach and both authors are warmly grateful to them and to the organizers.

\section{Twisted de Rham complexes for manifolds with boundary}
To set up the notation in the paper, we review the twisted de Rham
cohomology \cite{RW,BCMMS} with an odd-degree flux form and with
coefficients in a flat vector bundle.
We show that the twisted cohomology does not change under the scalings of
the flux form when its degree is at least $3$.
We also establish homotopy invariance, Poincar\'e duality and K\"unneth
isomorphism for these cohomology groups.

\subsection{Flat vector bundles, representations and Hermitian metrics}
\label{sect:scalar}
Let $X$ be a connected, compact, oriented smooth manifold.
Let $\rho\colon\pi_1(X)\to\GL(E)$ be a representation of the fundamental
group $\pi_1(X)$ on a vector space $E$.
The associated vector bundle $p\colon\E\to X$ is given by
$\E=(E\times\widetilde X)/\sim$, where $\widetilde X$ denotes the universal
covering of $X$ and $(v,x\gamma)\sim(\rho(\gamma)v,x)$ for all
$\gamma\in\pi_1(X)$, $x\in\widetilde X$ and $v\in E$.
If the representation $\rho$ is real or complex, so is the bundle $\E$,
respectively.
A smooth section $s$ of $\E$ can be uniquely represented by a smooth 
equivariant map $\phi\colon\widetilde X\to E$, {{so}} satisfying
$\phi(x\gamma)=\rho(\gamma)^{-1}\phi(x)$ for all $\gamma\in\pi_1(X)$
and $x\in\widetilde X$. 

Given any vector bundle $p\colon\E\to X$ over $X$, denote by $\Omega^i(X,\E)$
the space of smooth differential $i$-forms on $X$ with values in $\E$. 
A {\it flat connection} on $\E$ is a linear map
$$
\nabla^\E\colon\Omega^i(X,\E)\to\Omega^{i+1}(X,\E)
$$
such that 
$$
\nabla^\E(f\omega)=df\wedge\omega+f\,\nabla^\E\omega 
\qquad\text{and}\qquad(\nabla^\E)^2=0 
$$
for any smooth function $f$ on $X$ and any $\omega\in\Omega^i(X,\E)$.
If the vector bundle $\E$ is associated with a representation $\rho$
as {{above}}, an element of $\Omega^\bullet(X,\E)$
can be uniquely represented as a $\pi_1(X)$-invariant element in 
$E\otimes\Omega^\bullet(\widetilde X)$.
If $\omega\in\Omega^\bullet(\widetilde X)$ and $v\in E$,
then $v\otimes\omega$ is said to be $\pi_1(X)$-invariant
if $\rho(\gamma)v\otimes (\gamma^{{{-1}}})^*\omega=v\otimes\omega$
for all $\gamma\in\pi_1(X)$.
On such a vector bundle, there is a {\it canonical flat connection}
$\nabla^\E$ given, under the above identification, by
$\nabla^\E(v\otimes\omega)=v\otimes d\omega$, 
where $d$ is the exterior derivative on forms.

The usual wedge product on differential forms can be extended to
$$
\wedge\colon\Omega^i(X)\otimes\Omega^{j}(X,\E)\to\Omega^{i+j}(X,\E).
$$
Together with the evaluation map $\E\otimes\E^*\to\CO$, we have another product
$$
\wedge\colon\Omega^i(X,\E)\otimes\Omega^{j}(X,\E^*)\to\Omega^{i+j}(X).
$$
A Riemannian metric $g_X$ defines the Hodge star operator 
$$
\star \colon\Omega^i(X,\E)\to\Omega^{n-i}(X,\E),
$$
where $n=\dim X$.
A Euclidean or Hermitian metric $g_\E$ on $\E$ determines an $\RE$-linear
bundle isomorphism $\sharp\colon\E\to\E^*$, which extends to an $\RE$-linear
isomorphism 
$$
\sharp\colon\Omega^i(X,\E)\to\Omega^i(X,\E^*).
$$
One sets $\star_\E=\star\;\sharp=\sharp\;\star$ and for any 
$\omega,\omega'\in\Omega^i(X,\E)$, let
$$
(\omega,\omega')=\int_X\omega\wedge{{\star_\E}}\,\omega'.
$$
This makes each $\Omega^i(X,\E)$, $0\le i\le n$, a pre-Hilbert space. {{Unless otherwise specified, we shall simply denote $\star_\E$ by $\star$}}. 

When $\E$ is associated to an orthogonal or unitary representation $\rho$
of $\pi_1(X)$, $g_\E$ can be chosen to be compatible with the canonical
flat connection.

\subsection{Twisted de Rham cohomology for manifolds with boundary}\label{sect:twistedDR} 
{{Let now $Y$ be a compact connected oriented manifold with boundary and let again $H$ be a closed odd differential form on $Y$. }} Given a flat vector bundle $p\colon\E\to Y$ and an odd-degree,
closed differential form $H$ on $Y$, we set 
$\Omega^{\bar 0}(Y,\E)_r:=\Omega\even(Y,\E)$ with relative boundary conditions,
$\Omega^{\bar 1}(Y,\E)_r:=\Omega\odd(Y,\E)$ with relative boundary conditions,
$\Omega^{\bar 0}(Y,\E)_a:=\Omega\even(Y,\E)$ with absolute boundary conditions,
$\Omega^{\bar 1}(Y,\E)_a:=\Omega\odd(Y,\E)$ with absolute boundary conditions,
and $\nabla^{\E,H}:=\nabla^\E+H\wedge\,{{\bullet}}\;$.
{{We can assume without loss of generality and for simplicity that the closed form}} $H$ does not contain a $1$-form
component, which can be absorbed in the flat connection $\nabla^\E$. {{We denote by $\star$ the star Hodge operator associated with the orientation of $Y$ (compatible with that of $\partial Y$).
For a differential form $\omega$ on $Y$, the relative boundary condition is $B_r(\omega)=i^*\omega$, while the absolute 
boundary condition can be written using the orientation as $B_a(\omega)= i^*(\star\omega)$. Here $i^*$ is the restriction of a differential form to the boundary 
$i:\partial Y \hookrightarrow Y$. Notice that  if we write $\omega=\omega_0+\omega_1 dr$ in a collar neighborhood 
$Y_\epsilon\simeq \partial Y\times (-\epsilon, 0]$ of the boundary manifold 
$\partial Y$, then the relative boundary condition is $B_r(\omega)=\omega_0$ while the absolute 
boundary condition is  $B_a(\omega)= \omega_1$, \cite{GilkeyBook}. }}
For any $H$ as above, we define the {\it twisted relative de Rham cohomology groups of $\E$} as the quotients
$$
H^{\bar k}(Y, \partial Y,\E,H)=
\frac{\ker\,(\nabla^{\E,H}\colon\Omega^{\bar k}({{Y}},\E)_r\to
 \Omega^{\overline{k+1}}({{Y}},\E)_r)}
{\im\,(\nabla^{\E,H}\colon\Omega^{\overline{k+1}}({{Y}},\E)_r\to
 \Omega^{\bar k}({{Y}},\E)_r)},\quad k=0,1.
$$
{{When $H$ is chosen to satisfy the absolute boundary condition, }} we define the {\it twisted absolute de Rham cohomology groups of $\E$} as the quotients fulfils
$$
H^{\bar k}(Y, \E,H)=
\frac{\ker\,(\nabla^{\E,H}\colon\Omega^{\bar k}({{Y}},\E)_a\to
 \Omega^{\overline{k+1}}({{Y}},\E)_a}
{\im\,(\nabla^{\E,H}\colon\Omega^{\overline{k+1}}({{Y}},\E)_a\to
 \Omega^{\bar k}({{Y}},\E)_a)},\quad k=0,1.
$$
Here and below, the bar over an integer means taking the value modulo $2$. {{Notice that since the absolute cohomology is isomorphic 
to the cohomology of $Y$
we can always find a representative of the class of $H$ which satisfies the absolute boundary condition. 
Therefore and although the results of this paper can be stated for general $H$, we shall assume that $H$ fulfils the absolute boundary condition  so as to avoid
unnecessary technicalities  when dealing with a general closed odd degree form. }}

{{We shall see in the next subsection that the dimensions of the vector spaces}} $H^{\bar k}(Y,\E,H)$ and $H^{\bar k}(Y, \partial Y, \E,H)$  ($k=0,1$) are independent of the
choice of the Riemannian metric on $Y$ or the Hermitian metric on $\E$. 
The {\em twisted absolute Betti numbers} are denoted by 
$$
b_{\bar k, {{H}}}^a =b_{\bar k}(Y,\E,H):=\dim H^{\bar k}(Y,\E,H),\quad k=0,1,
$$
and the {\em twisted relative Betti numbers} are denoted by 
$$
b_{\bar k, {{H}}}^r =b_{\bar k}(Y, \partial Y,\E,H):=\dim H^{\bar k}(Y, \partial Y,\E,H),\quad k=0,1.
$$

\begin{lemma}\label{lem:fluxindep}
The absolute and relative twisted Betti numbers only depend on the cohomology class of the differential form $H$ in $H^{\bar 1}(Y)_a$.
\end{lemma}

\begin{proof}
{{Suppose $H$ is replaced by $H'=H-dB$ for some $B\in\Omega^{\bar 0}(Y)_a$,
then there is an isomorphism $\varepsilon_B:=e^B\wedge\cdot\,\colon
\Omega^\bullet(Y,\E)\to\Omega^\bullet(Y,\E)$ satisfying}}
$$
\varepsilon_B\circ\nabla^{\E,{{H'}}}=\nabla^{\E,{{H}}}\circ\varepsilon_B.
$$
{{Therefore the Poincar\'e lemma holds for the twisted differential when the
space is contractible. 
In general, $\varepsilon_B$ induces an isomorphism (denoted by the same)
\begin{align*}\label{e^B}
& \varepsilon_B\colon H^\bullet(Y,\E,H)\to H^\bullet(Y,\E,H')\\
& \varepsilon_B\colon H^\bullet(Y, \partial Y, \E,H)\to H^\bullet(Y,\partial Y,\E,H')
\end{align*}
on the twisted de Rham cohomology.}}
\end{proof}

\subsection{Homotopy invariance, products} Here we omit the proofs of all the statements as they are 
straightforward generalizations of the corresponding statements in the case of manifolds without boundary
\cite{MW,MW2} and the absolute and relative cohomology for manifold with boundary, cf. \cite{GilkeyBook,Schwarz}.

\begin{proposition}\label{prop:homotopy}
Given ${{Y}}$, $\E$ and $H$ as above, any smooth map $f\colon {{X\to Y}}$
of compact manifolds with boundary (implicitly assumed to map $\partial X$ to $\partial Y$) induces homomorphisms
$$
f^*\colon H^\bullet({{Y}},\E,H)\to H^\bullet({{X}},f^*\E,f^*H), \qquad 
f^*\colon H^\bullet(Y, \partial Y, \E,H) \to H^\bullet({{X}},\partial X, f^*\E,f^*H).
$$
Then the induced map $f^*$ depends only on the homotopy class of $f$.
\end{proposition}

\begin{corollary}\label{cor:hmtp}
Suppose $X$, $X'$ are smooth manifolds with boundary and $H$, $H'$ are closed, odd-degree
forms on $X$, $X'$, respectively. Let  $f\colon X\to X'$ be a smooth homotopy equivalence such that
$[f^*H']=[H]$, then $H^\bullet(X,H)\cong H^\bullet(X',H')$ and
 $H^\bullet(X,\partial X, H)\cong H^\bullet(X', \partial X', H').$
\end{corollary}

\subsection{Poincar\'e duality}

{{Let now $\E,\E'$ be flat vector bundles and let $H,H'$ be closed odd-degree
differential forms on a smooth compact manifold with boundary $X$.
Then the sequilinear product \label{lemma:cup}
\[   \Omega^{\bar k}(X,\E)_a\otimes\Omega^{\bar l}(X,\E')_r
{\longrightarrow}\Omega^{\overline{k+l}}(X,\E\otimes\E')_r  \]
given by 
$$
\omega \otimes \omega' \longmapsto \omega \wedge \bar \omega'
$$
induces a natural cup product
\[   H^{\bar k}(X,\E,H)\otimes H^{\bar l}(X, \partial X \E',H')
  \stackrel{\cup}{\longrightarrow}H^{\overline{k+l}}(X,\partial X, \E\otimes\E',H+\overline H'), \]
where $k,l=0,1$.
This is a consequence of the formula
\[  \nabla^{\E\otimes\E'\!,\,H+{{\bar {H'}}}}(\omega\wedge{{\bar{\omega'}}})=\nabla^{\E,H}\omega
    \wedge {{{\bar{\omega'}}}}+(-1)^{\bar k}\,\omega\wedge{{\bar{\nabla^{\E'\!,\,H'}\omega'}}},   \]
where $\omega\in\Omega^{\bar k}(X,\E)_a$, $\omega'\in\Omega^{\bar l}(X,\E')_r$.
When $\E'=\E^*$, $H'=-\overline H$, the cup product, composed with the pairing between
$\E$ and $\E^*$, takes values in $H^{\overline{k+l}}(X, \partial X)$. In particular, there is a natural sesquilinear cup product
\[   H^{\bar k}(X,\E,H)\otimes H^{\bar l}(X, \partial X,\E^*,-\overline H)
  \stackrel{\cup}{\longrightarrow}H^{\overline{k+l}}(X, \partial X).                \]
}}

\begin{proposition}[Poincar\'e duality]\label{prop:pd}
Let $X$ be an oriented {{(connected)}} compact manifold with boundary of dimension $n$.
Let $\E$ be a flat hermitian vector bundle on $X$. 
Suppose $H$ is a closed odd-degree differential form on $X$ which is boundary compatible.
Then, for $k=0,1$, there is a natural isomorphism
\[   H^{\bar k}(X,\E,H)\cong(H^{\overline{n-k}}(X, \partial X, \E^*,-\overline H))^*.   \]
\end{proposition}

\subsection{Twisted Kunneth theorem}

\begin{proposition}[K\"unneth isomorphism]\label{prop:ku}
Let $X_1$ be a smooth compact manifold with boundary, and let $X_2$ be a smooth closed manifold.
Suppose $\E_i$ are flat vector bundles over $X_i$ and $H_i$, closed
odd-degree forms on $X_i$, respectively, and $H_1$ is boundary compatible.
Let $\pi_i\colon X_1\times X_2\to X_i$ be the projections.
Set $\E_1\boxtimes\E_2=\pi_1^*\E_1\otimes\pi_2^*\E_2$ and
$H_1\boxplus H_2=\pi_1^*H_1+\pi_2^*{{\overline H_2}}$.
Then for each $k=0,1$,
there are natural isomorphisms
\begin{align*}
H^{\bar k}(X_1\times X_2,\E_1\boxtimes\E_2,H_1\boxplus H_2) 
& \cong\bigoplus_{l=0,1}H^{\bar l}(X_1,\E_1,H_1)
\otimes H^{\overline{k-l}}(X_2,\E_2,H_2),\\
H^{\bar k}(X_1\times X_2, \partial X_1 \times X_2,\E_1\boxtimes\E_2,H_1\boxplus H_2) 
& \cong\bigoplus_{l=0,1}H^{\bar l}(X_1, \partial X_1,\E_1,H_1)
\otimes H^{\overline{k-l}}(X_2,\E_2,H_2).   
\end{align*}
\end{proposition}

\subsection{Review of twisted Hodge theory}

Let $\cH^{\bar k}(X,\E,H)$ denote the nullspace of $\Delta^\E_H$ with absolute boundary conditions,
and $\cH^{\bar k}(X, \partial X, \E,H)$ denote the nullspace of $\Delta^\E_H$ with relative boundary conditions.
Then as a special case of Hodge theory of elliptic boundary value problems for manifolds with boundary, cf. \cite{GilkeyBook, Schwarz},

\begin{proposition}[Hodge isomorphism]\label{prop:hodge}
Let $X$ be an oriented {{(connected)}} compact manifold with boundary of dimension $n$.
Let $\E$ be a flat hermitian vector bundle on $X$. 
Suppose $H$ is a closed odd-degree differential form on $X$ which is boundary compatible.
Then, for $k=0,1$, the inclusion map induces isomorphisms
$$
\cH^{\bar k}(X,\E,H)  \cong  H^{\bar k}(X,\E,H), \qquad
\cH^{\bar k}(X, \partial X, \E,H)  \cong H^{\bar k}(X, \partial X, \E,H).
$$
\end{proposition}

\subsection{Scale invariance} 

{{Let again $X$ be a smooth compact manifold  with boundary $\partial X$}}. 

\begin{proposition}[Scale invariance]\label{scaling}
Let $\E$ be a flat vector bundle over {{$X$}} and $H$, an odd-degree closed
form on ${{X}}$ {{which satisfies the absolute boundary condition}}.
Suppose $H=\sum_{i\ge1}H_{2i+1}$, where each $H_{2i+1}$ is a $(2i+1)$-form.
For any $\lambda\in\CO\smallsetminus \{0\}$, let $H^{(\lambda)}=\sum_{i\ge1}\lambda^i H_{2i+1}$.
Then $H^\bullet(X,\E,H)\cong H^\bullet(X,\E,H^{(\lambda)})$
and  $H^\bullet(X, \partial X,\E,H)\cong H^\bullet(X,\partial X, \E,H^{(\lambda)})$.
\end{proposition}

\begin{proof}
For any $\lambda \in \CO$, let $c_\lambda$ act on $\Omega^\bullet({{X}},\E)$ 
with either relative or absolute boundary conditions, by 
multiplying $\lambda^{[\frac{i}{2}]}$ on $i$-forms.
Then $H^{(\lambda)}=c_\lambda(H)$ and
$c_\lambda\circ\nabla^{\E,H}=\lambda^k\;\nabla^{\E,H^{(\lambda)}}\circ
c_\lambda$ on $\Omega^{\bar k}({{X}},\E)$ for $k=0,1$.
If $\lambda\ne0$, then $c_\lambda$ induces the desired isomorphism on
twisted {{absolute and relative}} cohomology groups respectively.
\end{proof}

{{
\begin{corollary}
For any $\lambda$ in the unit circle, and denoting by $\Delta^{\E}_H$ the Laplacian operator associated with the odd degree closed form $H$, the following relation holds on relative forms as well as on absolute forms:
$$
c_\lambda\circ \Delta^\E_H \circ c_{\bar\lambda} = \Delta^\E_{H^{(\lambda)}}, \text{ where } c_\lambda = \lambda^{[i/2]} \text{ on $i$-forms}.
$$
\end{corollary}
}}

\begin{proof}
{{Recall that the hermitian scalar product of forms is given by 
$$
<\alpha, \beta> := \int_Y \alpha \wedge \star_\E\bar\beta.
$$}}
{{We have for any smooth forms $\omega, \omega'$ such that $(-1)^{deg \omega + deg \omega'} = -1$,
$$
<\nabla^{\E, H} \omega, \omega'> = -<\omega , \star_\E \nabla^{\E, -\overline H} \star_\E \omega' > + \int_{\partial Y} i^* (\omega \wedge \star_\E {\bar\omega'}),
$$
where $i:{{\partial X \hookrightarrow X}}$ in the inclusion of the boundary. If $\omega$ is a relative form then $i^* (\omega)=0$. If $\omega'$ is an absolute form then 
$ i^* (\star {\bar\omega'}) = 0$. Hence the adjoint ${\nabla^{\E, H}}^\dagger$ of $\nabla^{\E, H} $ on relative as well as on absolute forms is given by
$$
{\nabla^{\E, H} }^\dagger = - \star_\E\nabla^{\E, -\overline H} \star_\E.
$$
Thus, $\Delta_{\E,H} = - \star_\E \nabla^{\E, -\overline H} \star_\E  \nabla^{\E, H} - \nabla^{\E, H}\star_\E \nabla^{\E, -\overline H} \star_\E$. A straightforward inspection using Proposition \ref{scaling} then ends the proof.
}}

\end{proof}

\section{Twisted signature and the twisted signature complex}

Assume that $X$ is a closed oriented manifold of dimension $n=2m$ and that $H$ is an odd degree closed differential form on $X$. 
When $-\overline H = H$, i.e. when $H$ is pure imaginary, Poincare duality asserts that 
there is a nondegenerate sesquilinear pairing,
\[   
H^{\bar k}(X,\E,H) \otimes H^{\overline{n-k}}(X,\E^*, H) \longrightarrow \CO.   
\]
Since $\E$ is hermitian, there is a natural isomorphism $\E \cong \E^*$ given by 
the hermitian metric, so that the above becomes a sesquilinear form
\[   
{{B_0:}} H^{\bar k}(X,\E,H) \otimes H^{\overline{n-k}}(X,\E, H) \longrightarrow \CO.   
\]
{{In restriction to even forms we set $B=B_0$ while in restriction to odd forms we take $B=i B_0$ where $i=\sqrt{-1}$. In this way, we end up with a hermitian 
sesquilinear form $B$ on all differential forms,
which is a direct sum of the two hermitian forms. }}

{{
\begin{definition}\label{def:signature}
Let $H$ be any odd degree closed pure imaginary differential form on $X$. Then the  twisted signature ${\rm Sign}(X, \E, H)$ is by definition the signature of 
the hermitian  form $B$ on $H^*(X,\E,H)$.
\end{definition}
 }}

{{
\begin{corollary}
Let $H$ be as above a purely imaginary odd degree closed differential form on $X$. Then we define for any $\lambda \in \ST^1$, a sesquilinear hermitian form 
$B^\lambda$  by setting $B^\lambda = B_0^\lambda$ on even forms and $B^\lambda= i B_0^\lambda$ 
on odd forms, where for homogeneous forms
$$
B_0^\lambda (\omega, \omega'):= \lambda^{ m - deg (\omega)} \int_X \omega \wedge {\overline{\omega'}}.
$$
The form $B^\lambda$ induces a nondegenerate hermitian form  
on the twisted cohomology space $H^\bullet (X, \E, H^{{{(}}\lambda)})$ whose signature coincides with ${\rm Sign}(X, \E, H)$.
\end{corollary}
}}

\begin{proof}
That $B_0^\lambda$ is hermitian on even forms and skew-hermitian on odd forms is straightforward. It is also obvious from the definition that 
$$
B^\lambda (\omega, \omega') = B(c_{\bar\lambda} (\omega), c_{\bar\lambda} (\omega')) \text{ where } c_\mu (\alpha)=\mu^{[i/2]} \alpha \text{ for } \mu\in \C^*, deg (\alpha)=i.
$$
Therefore, if $\nabla^{\E, H^{(\lambda)}} \omega = 0$ and $\nabla^{\E, H^{(\lambda)}}\beta =\omega'$ then
$$
 \nabla^{\E, H} c_{\bar\lambda} \omega = 0 \text{ and } \nabla^{\E, H}({\bar\lambda}^k c_{\bar\lambda}\beta) = c_{\bar\lambda} \omega', k=0 \text{ or } 1.
$$
Now since $B$ descends to the cohomology of the differential $\nabla^{\E, H}$ we deduce that 
$$
B^\lambda (\omega, \omega') = 0.
$$
The non degeneracy of $B^\lambda$ is also deduced form the non degeneracy of $B$, which itself is a consequence of Poincar\'e duality as already mentioned. Moreover, the signature 
of $B^\lambda$ on $H^\bullet (X, \E, H^{\lambda)})$ coincides with that {{of}} $B$ on $H^\bullet (X, \E, H)$.
\end{proof}

\begin{remark}
 Notice that for purely imaginary $H$, and for $\lambda\in \ST^1$, we can write:
$$
- {\overline{H^{(\lambda)}}} = H^{(\bar\lambda)}.
$$
\end{remark}

In particular, 
we shall be concerned with the case $\lambda = i=\sqrt{-1}$  and we quote that in this case $H^\bullet(X,\E,H)\cong H^\bullet(X,\E,H^{(i)})$ and when $H$ is purely imaginary, 
the hermitian sesquilinear form $B^i$ is given 
in this case by
$$
B^i (\omega, \omega'):= (-1)^{[p/2]+m} i^m \times  \int_X \omega \wedge {\overline{\omega'}}, \omega \in \Omega^p (X, \E).
$$

{{
Consider the signature operator $B^\E_H = \nabla^{\E, H} + {\nabla^{\E, H}}^\dagger$ 
where $\nabla^{\E, H}=\nabla^\E+H\wedge\,\cdot\,$, 
$\nabla^\E$ is the  flat hermitian connection on $\E$, 
$ {\nabla^{\E, H}}^\dagger$ denotes its adjoint, 
then we have,}}

\begin{proposition}
In the notation above,
$$
B^\E_H \tau = - \tau B^\E_H
$$
if and only if $H = \sum i^{j+1} H_{2j+1} $ is an 
odd-degree closed differential form on $X$ and $H_{2j+1}$ is a real-valued differential form 
of degree ${2j+1}$. 
\end{proposition}

So, it is  in this case that one gets {{a twisted version of the usual}} signature operator, in contrast to the 
case of the twisted de Rham complex, cf. \cite{BCMMS,AS,MW,MW2}. 

\begin{proof} 
{{Suppose that $H = \sum i^{j+1} H_{2j+1} $ is an 
odd-degree closed differential form on $X$ and $H_{2j+1}$ is a real-valued differential form 
of degree ${2j+1}$, and  $\omega \in \Omega^{2k}(Y, \E)$,}}
\begin{align*}
& \tau (\nabla^{\E, H} + {\nabla^{\E, H}}^\dagger) \omega =  \tau \nabla^{\E, H}\omega - \tau \star {\nabla^{\E, -\overline H}}\star\omega \\
 & = i^{m-2k} \star \nabla^{\E}\omega + \sum_{j\ge 0} i^{j+1} \tau( H_{2j+1} \wedge \omega) 
 - \tau(\star \nabla^{\E} \star \omega) + \sum_{j\ge 0} {{(-i)}}^{j+1} \tau( \star H_{2j+1} \wedge \star \omega) \\
 & = i^{m+2k} \star \nabla^{\E} \omega +\sum_{j\ge 0}  i^{m+2(j+k)} i^{j+1} \star( H_{2j+1} \wedge \omega) 
 - i^{m-2k+2} \star^2 \nabla^{\E} \star \omega\\
 & \qquad + \sum_{j\ge 0}  i^{m-2(k-j-1)} {{(-i)}}^{j+1} \star^2 H_{2j+1} \wedge \star \omega \\
 & =  i^{m-2k} \star \left[ \nabla^{\E}  + \sum_{j\ge 0}  i^{-j+1}  H_{2j+1} \wedge\right]\omega
 -  i^{m-2k}  \left[ \nabla^{\E}  + \sum_{j\ge 0}  (-1)^{j+1} i^{-(j+1)}  H_{2j+1} \wedge\right]\star \omega\\
 & =   i^{m-2k} \left(\star \nabla^{\E, -\overline H} - \nabla^{\E, H} \star\right)
 \end{align*}
 On the other hand,
 \begin{align*}
(\nabla^{\E, H} + {\nabla^{\E, H}}^\dagger) \tau \omega & = \nabla^{\E, H}  \tau \omega - \star \nabla^{\E, -\overline H}
\star \tau \omega \\
& = {{i^{m-2k} \left( \nabla^{\E, H}  \star -  \star \nabla^{\E, -\overline H}\right)\omega}},
 \end{align*} 
A similar argument holds for odd degree forms $\omega$. More precisely, we get for a $2k+1$ degree form $\omega$:
$$
\tau (\nabla^ {\E, H} + {\nabla^{\E, H}}^\dagger )\omega ={{-i^{m-2k} (\nabla^{\E, H} \star + \star \nabla^{\E, -{\overline H}}) \omega}},
$$
while the other way computation gives exactly the opposite. 
This proves half the lemma. \\

 If conversely we assume that $(\nabla^{\E, H} + {\nabla^{\E, H}}^\dagger)  \tau + \tau (\nabla^{\E, H} + {\nabla^{\E, H}}^\dagger) = 0$ then restricting for instance to even forms we deduce by straightforward computation that for any $\omega \in\Omega^{2k} (Y, \E)$, the following identity holds
$$
K\wedge \star\omega + \star ({\overline{K}}\wedge \omega) =0,
$$
where $K= \sum_{j\geq 0} K_{2j+1} := \sum_{j\geq 0} (H_{2j+1} + (-1)^j  {\overline{H_{2j+1}}})$. Notice that for any $j, j'\geq 0$ 
$$
{\rm deg} (K_{2j+1} \wedge \star\omega ) = {{2m - 2k +2j+1}}\, \text{ while }\, {\rm deg} (* ({\overline{K}}_{2j'+1} \wedge \omega)) = {{2m -2k -2j'-1}}.
$$
Hence we deduce that for any $k\geq 0$ and any $\omega\in \Omega^{2k} (Y, \E)$,
$$
K\wedge \star\omega = 0 \text{ and } \star ({\overline{K}}\wedge \omega) =0.
$$
Taking for $\omega$ the {{$2m$}} form $\star 1$, we get out of the first relation
$$
K = 0 \text{ i.e. } H_{2j+1} + (-1)^j  {\overline{H_{2j+1}}} = 0, \quad \forall j\geq 0.
$$
\end{proof}

{{Thus if $H$ is purely imaginary, then the signature operator fits more naturally when we replace $H$ by $H^{(i)}$.}}

\begin{proposition}
{{ Assume  that the odd degree closed differential form $H$ is purely imaginary as above and that the dimension of $X$ is even, $n=2m$. Then the metric involution $\tau$ 
preserves the space 
${\mathcal H} (X, \E, H^{(i)})$ of harmonics corresponding to the differential $\nabla^{\E, H^{(i)}}$ and according to the $\pm 1$ eigenspaces, we have the decomposition
$$
{\mathcal H} (X, \E, H^{(i)}) = {\mathcal H}^+ (X, \E, H^{(i)}) \oplus {\mathcal H}^- (X, \E, H^{(i)}),
$$
and $\Sign (X, \E, H) = \dim {\mathcal H}^+ (X, \E, H^{(i)})- \dim {\mathcal H}^- (X, \E, H^{(i)})$.
}}
\end{proposition}

\begin{proof}
{{From the previous proposition, we deduce that since $H$ is pure imaginary,
$$
 B^{\E}_ {H^{(i)}} \circ \tau + \tau \circ B^{\E}_ {H^{(i)}} = 0.
$$ 
Therefore 
$$
\Delta^{\E}_ {H^{(i)}} \circ \tau = \tau \circ \Delta^{\E}_ {H^{(i)}},
$$
so that $\tau$ preserves the kernel $ {\mathcal H} (X, \E, H^{(i)})$ of $\Delta^{\E}_ {H^{(i)}}$ and we have the decomposition
$$
{\mathcal H} (X, \E, H^{(i)}) = {\mathcal H}^+ (X, \E, H^{(i)}) \oplus {\mathcal H}^- (X, \E, H^{(i)}).
$$
Now, assume that $\omega \in {\mathcal H}^+ (X, \E, H^{(i)})$, then writing $
\omega = \sum_{j\geq 0} \omega_j$, {{we get}} for any $0\leq j\leq m$, $\tau\omega_j = \omega_{2m -j}$, i.e. $\star \omega_j = (-i)^{m+ j(j-1)} \omega_{2m-j}$. Therefore,
\begin{eqnarray*}
 B^i (\omega, \omega) & = & i^m \times  \int_X \sum_{j} (-1)^{[j/2]+m} \omega_j \wedge {\overline{\omega_{2m-j}}}\\
& = & i^m i^m \sum_{j\geq 0} i^{j(j-1)} (-1)^{[j/2]+m} \int_X \omega_j \wedge \star {\overline{\omega_{j}}}\\
& = & \sum_{j\geq 0}\int_X  \omega_j \wedge \star {\overline{\omega_{j}}}\\
& = & <\omega, \omega>.
\end{eqnarray*}
The similar computation gives that $B^i$ is negative definite on ${\mathcal H}^- (X, \E, H^{(i)})$, which ends the proof since we know that $\Sign (X, \E, H)$ coincides with the signature
of the hermitian form $B^i$. }}
\end{proof}

\begin{corollary}
Given a purely imaginary odd degree closed differential form $H$ on $X$, the twisted signature $\Sign (X, \E, H)$ coincides with the index of the elliptic operator 
$$
B^\E_{H^{(i)}} : \Omega_+ (X, \E) \rightarrow \Omega_- (X, \E),
$$
where $\Omega_\pm (X, \E)$ are the (twisted by $\E$) differential forms $\omega$ which satisfy $\tau\omega = \pm \omega$.
\end{corollary}

\begin{proof}
 We simply apply the classical proof using Hodge theory for the elliptic complex associated with $(\Omega (X, \E), \nabla^{\E, H^{(i)}})$. See for instance \cite{BenameurHeitschJDG}.

\end{proof}

{{Notice that for purely imaginary $H$, the rescaled form $H^{(i)}$ is purely imaginary if and only if the degree of $H$ is congruent to $1$ modulo $4$. In this case
$$
H^{(i)}_{4j+1} = (-1)^j H_{4j+1}, \quad \forall j\geq 1,
$$
and the signature of $B^i$ on $H^\bullet (X, \E, H^{(i)})$ is then equal to the signature of $B$ on $H^\bullet (X, \E, H^{(i)})$.}}

{{
\begin{proposition}\
Assume that the closed odd degree form  $H$ is purely imaginary with degree congruent to $1$ modulo $4$. Then, when $m=2\ell$ is even,  the signature of the restriction of 
$B$ to odd forms is zero. When $m=2\ell -1$ is odd,  the signature of the restriction of $B$ to  even forms is zero. 
\end{proposition}
}}

\begin{proof}
{{In this case the twisted complex decomposes as }}

\hspace{1.4in}

\begin{picture}(415,100)
\put(115,75){$ \Omega ^{4\bullet}(X, \E) $}
\put(150,30){ $\vector(0,1){35}$}
\put(125,48){$\nabla^{\E,H}$}

\put(190,85){$\nabla^{\E,H}$}
\put(170,79){$\vector(1,0){60}$}

\put(195,25){$\nabla^{\E,H}$}
\put(230,19){$\vector(-1,0){50}$}

\put(235,75){$\Omega ^{4\bullet +1}(X, \E)$}
\put(255,65){ $\vector(0,-1){35}$}
\put(270,48){$\nabla^{\E,H}$}

\put(115,15){$\Omega ^{4\bullet +3}(X, \E)$}

\put(235,15){$\Omega^{4\bullet +2} (X,\E)$}
\end{picture}

{{Hence, the twisted odd and even cohomology spaces decompose into
\begin{multline*}
H^{even} (X, \E, H) = H^{4\bullet} (X, \E, H) \oplus H^{4\bullet + 2} (X, \E, H) \text{ and } \\H^{odd} (X, \E, H) = H^{4\bullet+1} (X, \E, H) \oplus H^{4\bullet + 3} (X, \E, H).
\end{multline*}
Assume for instance that $m=2\ell$ is even, then by direct inspection, we  see that the involution of $\Omega^{odd} (X, \E)$ given by 
$$
\sum_{k\geq 0} \omega_{2k+1} = \omega \longmapsto \sum_{k= 0}^{\ell-1} \omega_{2k+1} - \sum_{k=\ell}^{2\ell -1} \omega_{2k+1},
$$
interchanges the space of twisted harmonics which are in the $+1$-eigenspace of  $\tau$ with the space of  twisted harmonics which are in the $-1$-eigenspace 
of $\tau$. Therefore the signature of the hermitian sesquilinear form $B$ restricted to the odd forms in zero. The same involution on the even forms works {{to give the proof}} when $m$ is odd.}} 
\end{proof}

{{
\begin{proposition}
Let again $X$ be a smooth closed oriented manifold of even dimension $n=2m$ and let $\E$ be a flat hermitian bundle over $X$ and $H$, as before, a purely imaginary odd degree closed 
differential form on $X$. Then the twisted signature $\Sign(Y, \E, H)$ introduced in Definition \ref{def:signature} does not depend on $H$ 
and is given by
$$
\Sign(X, \E, H) = \Rank (\E) \times \Sign (X) = \Sign (X, \E),
$$
where $\Sign (X)$ is the signature of the closed oriented manifold $X$ (trivial if $m$ is odd). In particular, when $\E$ is the trivial line bundle, 
the $H$-twisted signature coincides 
with the usual signature of $X$.
\end{proposition}
}}

\begin{proof} {{As explained above, the signature $\Sign(X, \E, H)$ coincides with the signature of a corresponding hermitian sesquilinear form on the 
twisted cohomology corresponding 
to the closed odd degree form $H^{(i)}$. 
But this latter signature coincides with 
the index of the signature operator $B^\E_{H^{(i)}}: \Omega_+ (X, \E) \to \Omega_- (X, \E)$.  Now the principal symbol of $B^\E_{H^{(i)}}$ equals the principal symbol of 
$B^\E$, 
the usual signature operator with the local coefficients
$\E$.  Hence, since the index of $B^\E_H$ only depends on the (homotopy class of the) principal symbol, we deduce 
$$
\Sign(X, \E, H) = \Index (B^\E) = \Sign (X, \E).
$$
Applying for instance the Atiyah-Singer theorem \cite{AS3}, we know that 
$$
\Sign (Y, \E) = \Sign (Y) \times \Rank (\E).
$$}}
\end{proof}

\begin{remark}
If we decompose $B^\E$ into 
$(B_H^\E)^{even}: \Omega^{even}_+ (X, \E)\to \Omega^{odd}_- (X, \E)$ and to $(B_H^\E)^{odd}: \Omega^{odd}_+ (X, \E)\to \Omega^{even}_- (X, \E)$, then denoting by 
 $\chi (X, \E)$ the Euler characteristic with local coefficients $\E$ one gets the twisted analogue of well known relations,
\begin{multline*}
\Index ((B_H^\E)^{even}) = \frac{\Sign (X, \E) + \chi (X, \E)}{2}, \qquad \Index ((B_H^\E)^{odd}) = \frac{\Sign (X, \E) - \chi (X, \E)}{2}.
\end{multline*}
\end{remark}

We end this section by explaining very briefly how to extend our definition of the twisted signature to the case of manifolds with boundary. 
For even dimensional compact manifolds  $Y$ with boundary $\partial Y$, and for any odd degree purely imaginary closed differential form $H$ which satisfies the absolute boundary condition, 
we define the same hermitian sesquilinear form $B$ by $B= B_0$ on even forms and $B=iB_0$ on odd forms, where 
$$
B_0 (\omega, \omega' ) = \int_Y \omega \wedge {\overline{\omega'}}.
$$
But now $B$ is viewed on the relative cohomology 
groups $H^{\bar k}(Y, \partial Y, \E,H)$. The radical of $B$ is then precisely the 
kernel of the natural map from the relative cohomology  $H^{\bar k}(Y, \partial Y, \E,H)$ into 
the absolute cohomology $H^{\bar k}(Y,  \E,H)$. Therefore the
quadratic form is precisely nondegenerate on the 
image of the relative cohomology  $H^{\bar k}(Y, \partial Y, \E,H)$ inside the absolute cohomology 
$H^{\bar k}(Y,  \E,H)$. Again the signature of this hermitian sesquilinear form 
is defined to be the twisted signature ${\rm Sign}(Y, \E, H)$ of the manifold with boundary $Y$.

We may again  define the signature ${\rm Sign}(Y, \E, H)$ using the appropriate hermitian form for the rescaled form $H^{(\lambda)}$ corresponding to 
any $\lambda\in \ST^1$. This is particularly important for $\lambda = i$ as explained before, and where the hermitian form $B^i$ is given by the same formula
$$
B^i (\omega, \omega'):= (-1)^{[p/2]} (-i)^m \times  \int_Y \omega \wedge {\overline{\omega'}}, \omega \in \Omega^p (X, \E),
$$
except that now we integrate over the manifold with boundary $Y$.

\begin{corollary}
Let $\, \Sign(Y, \E, H)\,$ denote the signature of the quadratic form $B$ above. 
Since $\, B\, $ is defined in terms of the twisted relative cohomology groups, 
$\, \Sign(Y, \E, H)$ is a homotopy invariant of the quadruple $(Y, \partial Y, \E, [H])$.
\end{corollary}

\section{The APS type index theorem for the twisted signature complex}

\subsection{Eta invariants of twisted signature complexes}
\label{sect:etarho}

In this section, we define rho invariant $\rho(X,\E,H)$ and eta invariant $\eta(D^\E_H)$
of the twisted signature complexes. But let us review the boundary signature operator first.

{{Let $Y$ be an $2m$ dimensional compact, oriented Riemannian manifold with boundary, $\E$
a flat hermitian vector bundle over $Y$.  }}
 Consider the standard signature involution 
$\tau$ on $\E$-valued forms on $Y$, defined by $\tau:= i^{m+p(p-1)} {{\star_\E, \, \, \star_\E = \star \#}}$ where 
$\star$ is 
the Hodge star operator and $\#: \E \cong \E^*$ the isomorphism defined by the hermitian metric. {{For simplicity, we will usually replace $\star_\E$ by $\star$ when no confusion can occur.}}

{{ We denote by $i^*\omega$ the restriction of a closed differential form $\omega$ to the boundary $\partial Y$. We shall say that the given closed differential form 
$\omega$ on $Y$ is boundary-compatible if there exists a collar neighborhood
$Y_\epsilon\cong (-\epsilon, 0] \times \partial Y$ with projection map, 
 $p_\epsilon: \bar Y_\epsilon\cong [-\epsilon, 0] \times \partial Y \to \partial Y$ such that }}
\begin{equation}\label{bdry.compatible}
{{\omega \vert_{Y_\epsilon}  = p_\epsilon^* (i^*\omega).}}
\end{equation}

\begin{lemma}\label{lem:bdry.compatible}
Any closed differential form $\omega$ on $Y$ is cohomologous to a a closed differential form that is boundary compatible.
\end{lemma}

\begin{proof}
Define the maps,
\begin{align*}
q_\epsilon &: Y \longrightarrow Y\setminus Y_\epsilon\\
& q_\epsilon(y) = y, \qquad \forall y\in Y\setminus Y_\epsilon\\
& q_\epsilon(y) = p_\epsilon(y),  \qquad \forall y\in Y_\epsilon
\end{align*}
and the inclusion map,
$$
i_\epsilon :  Y\setminus Y_\epsilon \hookrightarrow Y.
$$
Note that for $y\in \partial(Y\setminus Y_\epsilon) \cong \partial Y$, one has $p_\epsilon(y)=y$, so that 
$q_\epsilon$ is continuous. It is standard to show that $q_\epsilon$ can be deformed to a smooth map
with similar properties, denoted by the same symbol.

Then $q_\epsilon\circ i_\epsilon =1$, so that $i_\epsilon^* q_\epsilon^*(\omega)=\omega$ for all 
$\omega \in \Omega^\bullet(Y\setminus Y_\epsilon).$ Since 
$q_\epsilon$ is a retraction from $Y$ to $Y\setminus Y_\epsilon$, it follows that
 there is a chain homotopy $K$ (cf. \cite{BoTu}) such that 
$$
q_\epsilon^*i_\epsilon^*(\omega)-\omega =(dK+Kd)\omega, \qquad \forall \omega  \in \Omega^\bullet(Y).
$$
This proves the lemma since $q_\epsilon^*i_\epsilon^*(\omega)$ is clearly boundary-compatible.
\end{proof}

\begin{lemma}\label{lem:bdry-iden}
Assuming that the 
Riemannian metric on $Y$ and the hermitian metric on $\E$ is of product type near the boundary {{and that the closed form $H$ is  boundary-compatible}},
we explicitly identify 
the twisted signature operator near the boundary
to be 
$B^\E_H = \sigma\left(\frac{\partial}{\partial r} + D^\E_H \right),$ where 
$r$ is the coordinate in the normal direction, $\sigma$ is a bundle isomorphism and 
finally, the self-adjoint elliptic operator {{$D^\E_H$ is given on 
$\Omega^{p}(\partial Y, \E)$ by }}
$$
D^\E_H = i^{m+{{p(p+1)}}} ( \nabla^{\E, H}_{\partial Y} \star - (-1)^p \star \nabla^{\E, -\overline H}_{\partial Y}).
$$
\end{lemma}

\begin{proof}
{{We use the identification}}
 $$
{{ \Omega^{even/odd}(Y_\epsilon, \E) \cong \Omega^{even/odd}(\partial Y, \E) \oplus 
\Omega^{odd/even}(\partial Y, \E) \wedge dr}}
$$
 {{where $Y_\epsilon \cong \partial Y \times [0, \epsilon)$
 and the variable
in $ [0, \epsilon)$ is $r$.  Under this identification, one computes the twisted signature operator $B^\E_H$ 
 to be the matrix}}
$$
B^\E_H = {\frac{\partial}{\partial r} \qquad \qquad \qquad \nabla^{\E, H}_{\partial Y} + {\nabla^{\E, H}_{\partial Y}}^\dagger
 \choose \nabla^{\E, H}_{\partial Y} + {\nabla^{\E, H}_{\partial Y}}^\dagger \qquad\qquad -\frac{\partial}{\partial r} }
$$
{{An inspection of the metric grading allows to define an identification $J^\pm$ of $\Omega^{*} (\partial Y, \E)$ with $\Omega^{\pm}(Y_\epsilon, \E)$. More precisely, we set}}
$$
{{J^\pm  (\alpha) := \alpha \pm (i)^{m+p(p-1)} \star_{\partial Y} \alpha \wedge dr, \quad \alpha\in \Omega^{p} (\partial Y, \E).}}
$$
{{Hence, the  computation of the boundary operator $D^\E_H={{(J^-)^{-1} B^\E_H J^+}}$ gives the announced formula.}}
\end{proof}
{{Notice that on $\Omega^{2h}(\partial Y, \E)$ for instance we have $
D^\E_H = i^m(-1)^{h} (\nabla^{\E, H}_{\partial Y} \star - \star \nabla^{\E, -{\overline H}}_{\partial Y} ).$ 
}}

\begin{remark}
{{The operator $D_H^\E$ is self-adjoint and preserves the parity of the forms. Moreover, if we set $T:=i^{m+p(p+1)} \star$ on $\Omega^p (\partial Y, \E)$ then
$
T^2= 1 \text { and } T\circ D_H^\E\circ  T = D_H^\E.
$ 
So, it suffices to concentrate on the restriction of $D_H^\E$ to the even forms.}}
\end{remark}

{{Let us digress now to recall the eta invariant.}}
Given a self-adjoint elliptic differential operator $A$ of order $d$ on a closed 
oriented manifold $X$ of dimension $2m-1$,
the {\em eta-function} of $A$ is
$$
\eta(s,A):=\Tr'(A|A|^{-s-1}), 
$$
where $\Tr'$ stands for the trace restricted to the subspace orthogonal to
$\ker(A)$. By \cite{APS1, APS2, APS3}, $\eta(s,A)$ 
is holomorphic when $\Re(s)>n/{{d}}$ and can
be extended meromorphically to the entire complex plane with possible simple
poles only,  and is
known to be holomorphic  at $s=0$, see for instance \cite{APS1, BF} for Dirac type operators, and the generalization to any such $A$ given in \cite{GilkeyBook}, Lemma 4.3.5. The {\em eta-invariant}
of $A$ is thus defined as
$$
\eta(A) = \eta(0, A).
$$
Replacing $A$ by $\widetilde A= A + P_A$, where $P_A$ denotes the orthogonal projection
to the nullspace of $A$, then $\widetilde A$ is a self-adjoint elliptic pseudodifferential 
operator which is invertible and we can express the {eta-function} of $A$ as
$$
\eta(s,A):=\Tr(A|\widetilde A|^{-s-1})
$$
and so the eta invariant of $A$ can {{also}} be defined as being the regularized trace
$$
\eta(A) = \Tr_Q(A |\widetilde A|^{-1})
$$
where $Q=  |\widetilde A|$ is a positive elliptic pseudodifferential 
operator which is invertible{{, cf. \cite{CDP}}}.

The eta-function is related
to the heat kernel by a Mellin transform
$$
\eta(s,A)=\frac{1}{\Gamma(\frac{s+1}{2})}\int_0^\infty t^{\frac{s-1}{2}}\Tr(Ae^{-t{\widetilde A}^2})\,dt.
$$
Then by Lemma 4.3.2 in \cite{GilkeyBook} and also \cite{APS3} one has

\begin{theorem}[\cite{BF, GilkeyBook}]\label{zeta-hol}
Let $A$ be a first order elliptic operator. Then the eta function $\eta(s,A)$ of $A$ has a meromorphic continuation 
to the complex plane with no pole at $s=0$. The eta invariant $\eta(A)$ of $A$ is then defined as $\eta(0,A)$.
\end{theorem}

Let 
$H = \sum_{j\geq 1} i^{j+1} H_{2j+1} $ an 
odd-degree closed differential form on $X$ where $H_{2j+1}$ is a real-valued differential form 
of degree ${2j+1}$. Denote by $\E$ a hermitian flat vector bundle over $X$ with the 
canonical flat connection $\nabla^\E$.
Consider the twisted odd signature operator 
${{D^\E_H = i^m (-1)^{h} (\nabla^{\E, H} \star -  \star \nabla^{\E, -\overline H})}}$ on 
$\Omega^{2h}(X, \E)$, acting on even degree forms. Then $D^\E_H$ is a self-adjoint
elliptic operator 
and let $\eta(D^\E_H )$ denote its eta invariant. 

\begin{definition}\label{DefRho}
The {\em twisted rho invariant} 
$\rho(X, \E, H)$ is defined to be, $\eta(D^\E_H ) - \Rank(\E)\, \eta(D_H )$, where 
$D_H$ is the odd signature operator on $\Omega^{even}(X)$ {{corresponding to $\E$ replaced by the trivial line bundle}}.
\end{definition}

Although the eta invariant  $\eta(D^\E_H )$  is only a spectral invariant, 
the twisted rho invariant $\rho(X,\E,H)$ is independent of the
choice of the Riemannian metric on $X$ and the Hermitian metric on $\E$.
The twisted rho invariant $\rho(X,\E,H)$ is also invariant if $H$ is deformed within its cohomology class.
These results will be established in section \ref{sect:var}.

\subsection{The twisted signature formula}
\label{sect:twistedsig}

The goal of this section is to prove the analogue of the Atiyah-Patodi-Singer index theorem for the 
twisted signature complex with non-local boundary conditions. The central result, Corollary \ref{cor:sig},
is one of the key tools that will be used to prove our results on the eta and rho invariants for the twisted odd signature
operator in the following section.

$D^\E_H$ is an elliptic self-adjoint operator, and 
by \cite{APS3}, the non-local boundary condition given by $P^+(s\big|_{\partial Y})=0$
where $P^+$ denotes the orthogonal projection onto the eigenspaces with positive eigenvalues,
then the pair $(B^\E_H; P^+)$ is an elliptic boundary value problem. 
By the Atiyah-Patodi-Singer index theorem 3.10 \cite{APS1} and its extension in \cite{GilkeyBook}, we have

\begin{proposition}\label{prop:sig1}
In the notation of Lemma \ref{lem:bdry-iden} above,
$$
\Index(B^\E_H; P^+) = \Rank(\E) \int_Y \alpha_0(x)dx - \left(\dim(\ker(D^\E_H))
+\eta(D^\E_H)\right),
$$
where $\eta(D^\E_H)$ denotes the eta invariant of the odd signature operator $D^\E_H$ acting on even degree
forms on the boundary. Here $ \Rank(\E) \alpha_0(x)$ is the constant term in the asymptotic expansion (as $t\to 0^+$) 
of the pointwise supertrace, $\tr_s(\exp(-t{B^\E_H}^2)(x, x))$.
\end{proposition}

\begin{remark}\label{rem:sig1} The precise form of $\alpha_0(x)$ is unknown for general $H$. However in the case when $H=0$, the local index theorem cf. \cite{BGV} establishes that the Atiyah-Hirzebruch
$\LL$-polynomial applied to the curvature of the Levi-Civita connection, wedged by the Chern character of the flat bundle 
$\E$, is equal to
 $\alpha_0(x)$ is equal to  times the rank of $\E$.
   In the case when degree of $H$ is equal to 3, it  follows from a result of Bismut \cite{Bismut} that the Atiyah-Hirzebruch
$\LL$-polynomial applied to the curvature of a Riemannian connection defined in terms of the 
Levi-Civita connection together with a torsion tensor determined by $H$, wedged by the Chern character of the flat bundle 
$\E$,  is equal to
 $\alpha_0(x)$ times the rank of $\E$. The proof that in general, one also gets
  $\alpha_0(x)$ times the rank of $\E$ is contained in the Appendix to this paper, and in particular,
Corollary
\ref{cor:indexcoeff}.
\end{remark}

Let $\hat Y = Y \cup \partial Y \times (-\infty, 0]$ be the original manifold $Y$ with a cylinder attached, 
and endowed with the natural smooth manifold structure, as schematically depicted below.

\begin{figure}[h]
\includegraphics[height=3.5in]{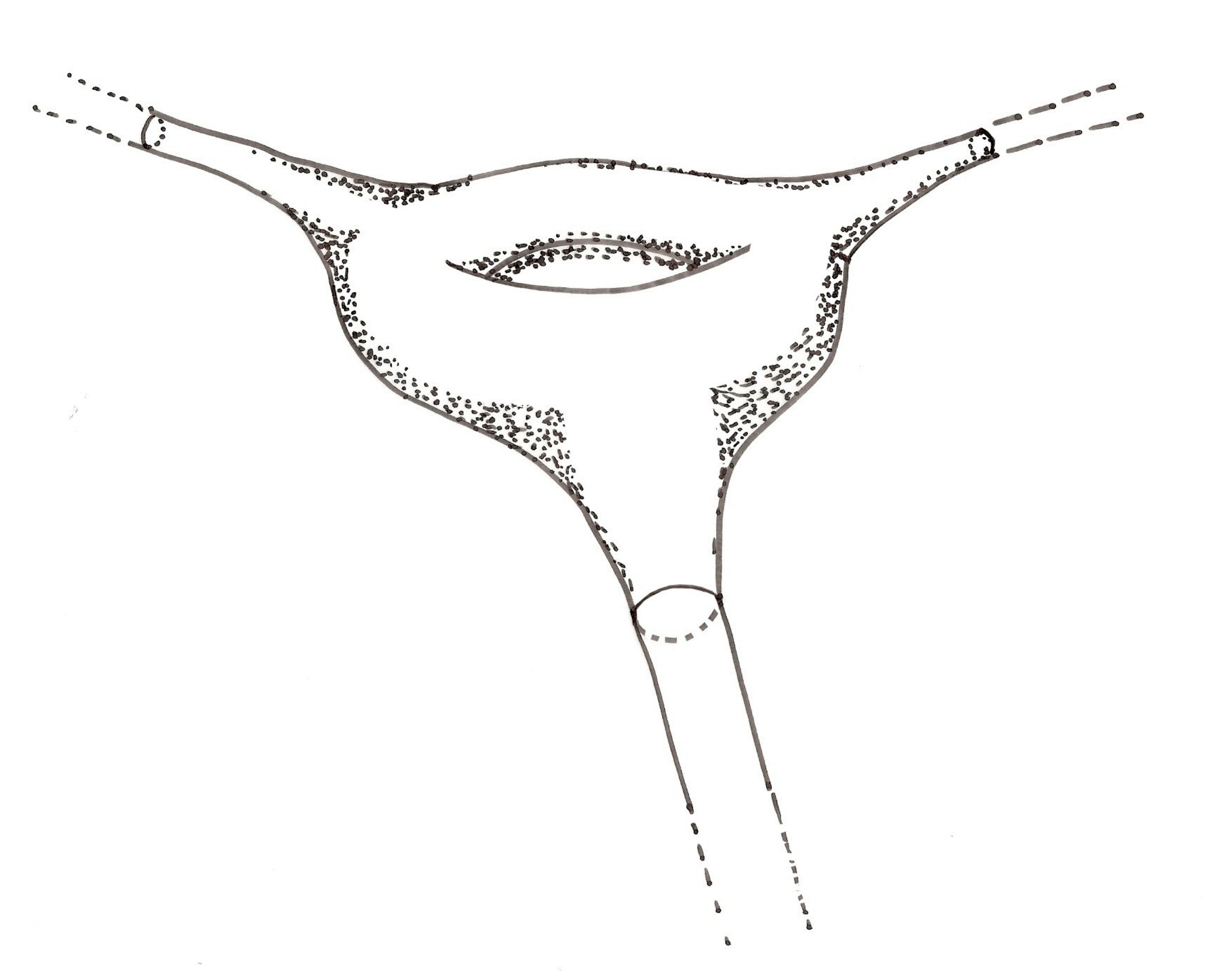}\\
\caption{$\hat Y$}
\label{fig:cylinder}
\end{figure}

The metric on the cylinder $ \partial Y \times (-\infty, 0]$ is $g_{\partial Y} + dr^2$, which is glued smoothly to the 
metric on $Y$. The bundle $\E$ is extended by pullback to the cylinder $ \partial Y \times (-\infty, 0]$ and denoted
$\hat\E$, together 
with the hermitian metric. {{Since the closed odd form $H$ is boundary-compatible, it is also smoothly extended to a closed bounded odd degree form $\hat H$ on $\hat Y$. }}
Besides $L^2$ sections of $\hat \E$-valued differential forms on $\hat Y$, we will also
be interested in {\em extended}  $L^2$ sections of $\hat \E$-valued differential forms on $\hat Y$, by which we
mean a form which is locally in $L^2$ and such that for large negative $r$, 
$$
f(y, r) = g(y, r) + f_\infty(y)\qquad \qquad y \in \hat Y
$$
where $g$ is in $L^2$ and $f_\infty \in \ker(D^\E_H)$. That is, $f$ has $f_\infty$ as asymptotic value as $r\to-\infty$. 

The following is a special case of  Proposition (3.11) in \cite{APS1}. 

\begin{proposition}\label{prop:sig2}
$(B^\E_H; P^-)$ is the adjoint of $(B^\E_H; P^+)$,  where $P^-$ denotes the orthogonal projection onto the 
eigenspaces with negative eigenvalues, and 
\begin{enumerate}
\item $\ker(B^\E_H; P^+)$ is isomorphic to the space of $L^2$-solutions of ${{B^{\hat\E}_{\hat H}}} s=0$ on $\hat Y$;
\item $\ker(B^\E_H; P^-)$ is isomorphic to the space of extended $L^2$-solutions of ${{{B^{\hat\E}_{\hat H}}}}^\dagger s=0$ on $\hat Y$.
\end{enumerate}
\end{proposition}

By Corollary (3.14) of \cite{APS1}, we have

\begin{proposition}\label{prop:sig3}
$\Index(B^\E_H; P^+) = h^+(\hat\E, \hat H) - h^-(\hat\E, \hat H) - h_\infty(\E)$, where $h^+(\hat\E, \hat H) $
is the dimension of the space of $L^2$-solutions of $({{B^{\hat \E}_{\hat H})}}^+ s=0$ on $\hat Y$, $h^-(\hat\E, \hat H) $
is the dimension of the space of extended $L^2$-solutions of $({{B^{\hat \E}_{\hat H})}}^- s=0$ on $\hat Y$, $h_\infty(\E)$
is the dimension of the nullspace of $D^\E_H$ consisting of limiting values of extended $L^2$-sections satisfying the adjoint equation. 
\end{proposition}

By Proposition (3.15) of \cite{APS1}, one has,

\begin{proposition}\label{prop:sig4}
The $L^2$-solutions of $({{B^{\hat \E}_{\hat H}}})^+$ and {{$(B^{\hat \E}_{\hat H}){{^-}} (B^{\hat \E}_{\hat H}){{^+}}$}} coincide on $\hat Y$,  and the same is true 
for the extended $L^2$-solutions. There is a similar coincidence when {{ $(B^{\hat \E}_{\hat H})^+$ replaced by $(B^{\hat \E}_{\hat H})^-$}}.
\end{proposition}

We will soon give topological interpretations of  $h^\pm(\hat\E, \hat H)$  and $h_\infty(\E)$. {{Recall that $H$ is assumed to be boundary-compatible.}}

\begin{proposition}\label{prop:sig5}
The space ${\mathcal H}^\bullet (\hat Y, \hat \E, \hat H)$ of twisted $\hat \E$-valued $L^2$ harmonic forms on $\hat Y $
is naturally isomorphic to the image ${\hat H} (\hat Y, \hat \E, \hat H)$ of the compactly supported cohomology $H^\bullet_c(\hat Y, \hat \E, \hat H)$ inside 
the cohomology $H^\bullet(\hat Y, \hat \E, \hat H)$, or equivalently, the image of the relative cohomology
 $H^\bullet(Y, \partial Y, \E, H)$ inside the absolute cohomology $H^\bullet(Y,  \E, H)$.
\end{proposition}

\begin{proof} 

The proof of Proposition (4.9) in \cite{APS1} extends with minor modifications to our  setting. By Proposition \ref{prop:sig4}, the ($\E$-valued twisted) $L^2$ harmonic forms on $Y$ coincide with the $L^2$ nullspace 
 of  $\nabla^{{\hat \E},{\hat H}} + ({\nabla^{{\hat \E}, {\hat H}}})^\dag$ acting on $\hat Y$. 
Despite the untwisted case, here the twisted Laplacian does not preserve the degree of the forms.

Notice that an ${\hat H}$-twisted  $L^2$ harmonic form $\omega$ is in the intersection of the nullspaces of $\nabla^{\hat E, {\hat H}}$ and its adjoint. 
In particular,  $\omega$ defines an element $[\omega]$ of $H^\bullet(\hat Y, \hat \E; {\hat H})$. To see that $[\omega]$ lies in the image of
$H^\bullet_c(\hat Y, {\hat H})$, it is sufficient to check that its restriction to $H^\bullet(\partial Y; H_\partial)$ is zero. Now given any $\omega'$ in the null space 
of $\nabla^{\E^*, -{\overline H}}$ over $\partial Y$ and any $u>0$, we denote by $\omega'_u$ the corresponding section over $\partial Y \times \{u\}$ and by $i_u: \partial Y\times \{u\}\hookrightarrow {\hat Y}$ the inclusion. Then we can write 
$$
\int_{\partial Y \times \{0\}} i_0^*\omega \wedge {\overline{\omega'_0}} = \int_{\partial Y \times \{u\}} i_u^*\omega \wedge {\overline{\omega'_u}}.
$$
But since $\omega$ is an ${\hat H}$-twisted  $L^2$ harmonic form, $i_u^*\omega$ is exponentially decreasing as $u\to +\infty$ (see (3.17) in \cite{APS1}) while $\omega'_u$ is constant. Hence, we see that
$$
\int_{\partial Y \times \{0\}} i_0^*\omega \wedge {\overline{\omega'_0}} = 0.
$$ 
Applying Poincar\'e duality on $\partial Y$ with the restricted closed odd form ${{H_\partial=}}i^*H$ to the boundary (Proposition \ref{prop:pd}), we deduce that 
$i_0^*[\omega] = 0$ in $H^\bullet (\partial Y, \hat E; H_\partial)$ and henceforth that the map $\omega \mapsto [\omega]$ yields
$$
\alpha_H: {\mathcal H}^\bullet (\hat Y, \hat \E, \hat H) \rightarrow {\hat H} (\hat Y, \hat \E, \hat H).
$$
We now prove as in \cite{APS1} that $\alpha_H$ is actually an isomorphism. Surjectivity is classical and uses a twisted version of the de Rham-Kodaira theorem, see page 140 and
 relations (1) page 105 in \cite{deRham} which obviously extend to our twisted situation. Recall that ${\hat H}$ is a bounded form 
and suppose now that $\omega$ is a twisted harmonic even form such that $\alpha_H \omega =0$. Then $\omega$ is exponentially decaying and using the compactification of $\hat Y$ obtained by adding a copy of $\partial Y$ at infinity, we deduce exactly as \cite{APS1} that there exists a bounded odd form $\theta$ with coefficients in $\hat \E$ such that $\omega=\nabla^{\hat\E, {\hat H}} \theta$. More precisely, $\omega= \omega_0+\omega_1 dr$, $\theta=\theta_0+\theta_1 dr$ and $\omega_0$ and $\omega_1$ are exponentially decaying with the variable $r$ while $\theta_0$ and $\theta_1$ are bounded. We then consider for any $U>0$ the compact submanifold ${\hat Y}_U$ of $\hat Y$ corresponding to the radial variable $r\leq U$. We then compute using that $H$  graded commutes with all the forms with coefficients in $\hat \E$
\begin{eqnarray*}
<\omega, \nabla^{\hat\E, {\hat H}}\theta >_{{\hat Y}_U} & = & <\omega, \nabla^{\hat\E}\theta>_{{\hat Y}_U} + <\omega, H\wedge \theta>_{{\hat Y}_U}\\
& = & <(\nabla^{\hat\E})^\dag\omega, \theta>_{{\hat Y}_U} + <\omega_1, \theta_0>_{\partial ({\hat Y}_U)} + <\omega, H\wedge \theta>_{{\hat Y}_U}\\
&=&  <\omega_1, \theta_0>_{\partial ({\hat Y}_U)} + <(\nabla^{\E})^\dag\omega, \theta>_{{\hat Y}_U} + <\star {\overline H} \wedge \star \omega, \theta>_{{\hat Y}_U}
\end{eqnarray*}
The last equality being {{an easy consequence of}} the definition of the hermitian scalar product. Recall that $(\nabla^{\E, {\hat H}})^\dag = -\star \nabla^{\hat\E, -{\overline{{\hat H}}}}  \star$, therefore,  we get the following (twisted Green formula) 
$$
<\omega, \nabla^{\hat\E, {\hat H}} \theta >_{{\hat Y}_U} =  <\omega_1, \theta_0>_{\partial ({\hat Y}_U)}  + <(\nabla^{\hat\E, {\hat H}})^\dag\omega, \theta>_{{\hat Y}_U} .
$$
Recall on the other hand that $\omega$ is harmonic, so $(\nabla^{\E, {\hat H}})^\dag\omega =0$ and also that $\nabla^{\hat\E, {\hat H}} \theta = \omega$, so we get
$$
<\omega, \omega >_{{\hat Y}_U} =  <\omega_1, \theta_0>_{\partial ({\hat Y}_U)} .
$$
Again because $ \theta_0$ is bounded and $\omega_1$ is exponentially decaying, we deduce that as $U$ goes to infinity $
<\omega, \omega >_{{\hat Y}} = 0.$

\end{proof}

Recall that ${\rm Sign}(Y, \E, H)$ denotes the signature of a natural 
quadratic form on the image 
of the twisted relative cohomology $H^\bullet(Y, \partial Y, \E, H)$ inside the twisted absolute 
cohomology $H^\bullet(Y,  \E, H)$. 
Also let $\cH^\pm(\hat Y, \hat \E, \hat H) = \cH^\bullet (\hat Y,  \hat \E, \hat H) \cap \Omega^\pm(\hat Y, \hat \E)$.

\begin{proposition}\label{prop:sig6} In the notation of Proposition \ref{prop:sig3} above,
we have $h^+(\E, H)_\infty = h^-(\E, H)_\infty = h$, where $h = \dim(\ker(D^\E_H)$.
\end{proposition}

As an immediate consequence of Proposition \ref{prop:sig1}, Proposition \ref{prop:sig3} and 
Proposition \ref{prop:sig6}, we have the identification,

\begin{theorem}\label{thm:sig}
In the notation above, 
$$
\Index(B^\E_H; P^+) + \dim(\ker(D^\E_H))= 
{\rm Sign}(Y, \E, H)
$$
\end{theorem}

Now 
a  consequence of Proposition \ref{prop:sig1} and Theorem \ref{thm:sig}, we
have,

\begin{corollary}\label{cor:sig}
In the notation of Proposition \ref{prop:sig1} above, 
\begin{equation}\label{sign-thm}
{\rm Sign}(Y, \E, H)= {{\Rank (\E)}} \int_Y \alpha_0(x)dx - \eta(D^\E_H).
\end{equation}
\end{corollary}

This is the main tool used to establish our results about the twisted rho invariant. 
Via Corollary \ref{cor:sig}, we see that 

\begin{corollary}\label{cor:rho}
If $X = \partial Y$, with $\E$ and $H$ extending to $Y$ as $\widetilde\E$ and $\widetilde H$ respectively, 
then $$\rho(X,\E,H) = \Rank(\E)\,
{\rm Sign}(Y, \widetilde H) - {\rm Sign}(Y, \widetilde\E, \widetilde H).$$
\end{corollary}

This corollary says for instance that the rho invariant is an integer in this particular situation.

\section{Stability properties of the twisted rho invariant}
\label{sect:var}

Here we prove {{that } the twisted rho invariant $\rho(X,\E,H)$ is independent of the
choice of the Riemannian metric on $X$ and the Hermitian metric on $\E$ needed
in its definition. It is also invariant if $H$ is deformed within its cohomology class.
The proofs of these results rely on the index theorem for twisted signature 
complexes for manifolds with boundary, established in section \ref{sect:twistedsig}.
We also state and prove the basic functorial properties of the rho invariant
 for the twisted signature complex

\subsection{Variation of the Riemannian and Hermitian metrics}
We assume that $X$ is a compact oriented manifold of odd dimension.
Let $g_X$ be a Riemannian metric on $X$ and $g_\E$, an Hermitian metric 
on $\E$.
Suppose that the pair $(g_X,g_\E)$ is deformed smoothly along a one-parameter
family with parameter $t\in\RE$, then the operators $\ast$, $\sharp$ and 
$\star_\E=\ast\sharp=\sharp\ast$ (see \$\ref{sect:scalar}) all depend smoothly
on $t$.
We have,

\begin{theorem}[metric independence of the rho invariant]\label{thm:indept}
Let $X$ be a compact, oriented manifold of odd dimension, $\E$, a flat hermitian vector
bundle over $X$, and $H = \sum i^{j+1} H_{2j+1} $ is an 
odd-degree closed differential form on $X$ and $H_{2j+1}$ is a real-valued differential form 
of degree ${2j+1}$. 
Then the rho invariant $\rho(X,\E,H)$ of the twisted signature complex
does not depend on the choice of the Riemannian metric on $X$ or the
Hermitian metric on $\E$. 
\end{theorem} 

\begin{proof}
Consider the manifold with boundary $Y = X \times [0, 1]$, where the boundary $\partial Y 
= X \times \{0\} - X \times \{1\}$. Let $g_0, g_1$ denote two Riemannian metrics on $X$ and 
$\lambda_0, \lambda_1$ denote two flat hermitian metrics on $\E$. Let $f: [0, 1] \to  [0, 1]$
be a smooth function such that $f (t)= 0$ for all $t$ small and $f (t)= 1$ for all $1-t$ small.
Define $g_t = (1-f(t))g_0 + f(t) g_1$, $\lambda_t =  (1-f(t))\lambda_0 + f(t) \lambda_1$. Then 
$g = g_t + dt^2$ defines a Riemannian metric on $Y$ which is of product type near the boundary,
and $\lambda = \lambda_t + dt^2$ defines a hermitian metric on $\pi^*(\E)$ (where $\pi : Y \to X$
denotes projection onto the first factor) which is also of product type near the boundary. Clearly 
$\pi^*(H)$ is a boundary compatible, closed odd degree form on $Y$.

\begin{figure}[h]
\includegraphics[height=3in]{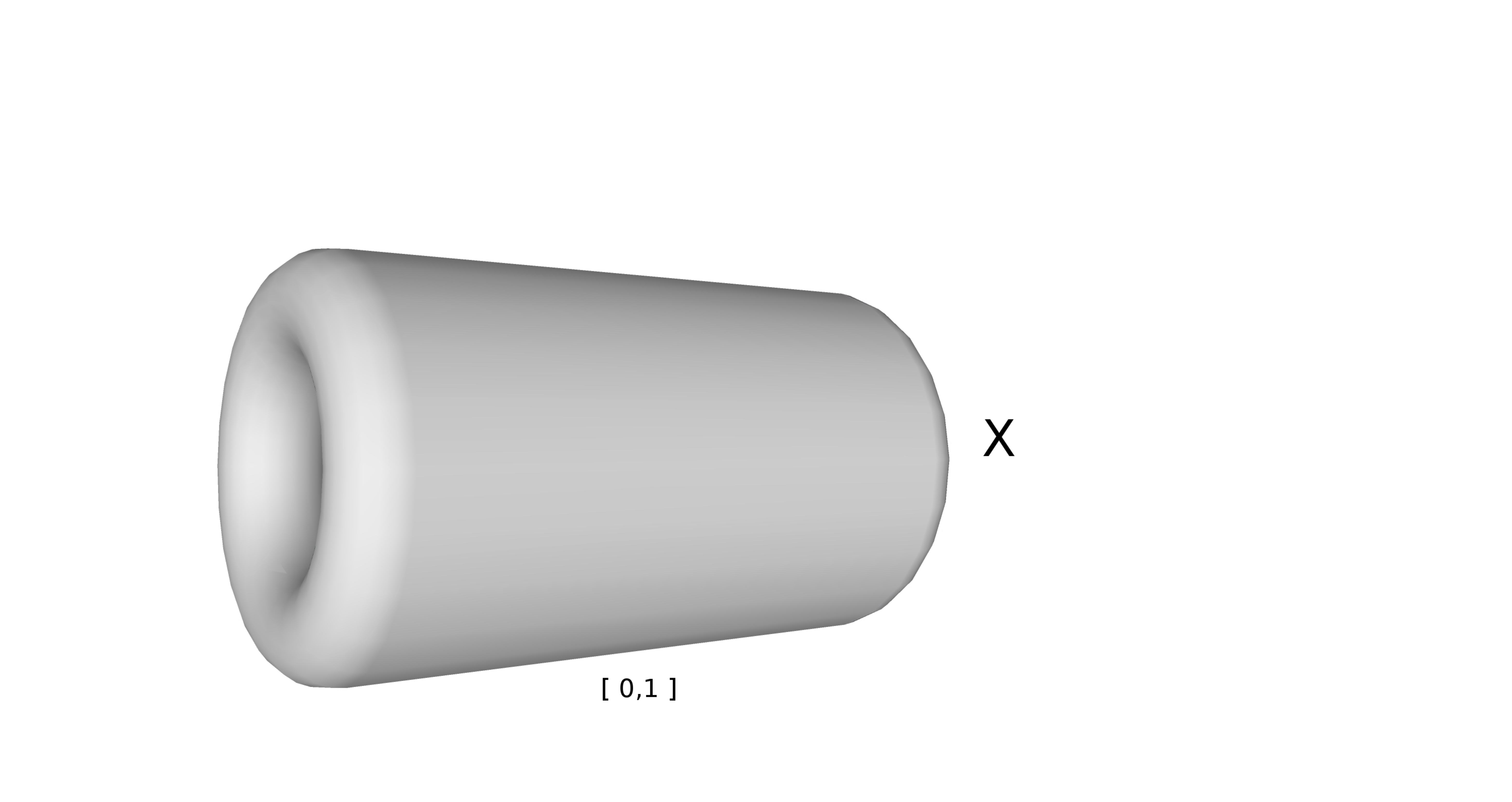}\\
\caption{Cylindrical manifold, $\quad g= g_t + dt\otimes dt,  \lambda = \lambda_t + dt\otimes dt$}
\label{fig:box-metric}
\end{figure}

Applying the signature theorem for the twisted signature complex, Corollary \ref{cor:sig}, we get
\begin{equation}\label{eqn:sig1}
{\rm Sign}(X \times [0,1], \pi^*(\E), \pi^*(H))=\Rank(\E) \int_Y \alpha_0(x)dx - \left(\eta(D^\E_H, g_1, \lambda_1)
-\eta(D^\E_H, g_0, \lambda_0)\right)\\
\end{equation}
We next {{show}} that ${\rm Sign}(X \times [0,1], \pi^*(\E), \pi^*(H))=0$. This follows from the fact that
the image of the relative twisted cohomology $H^\bullet(Y, \partial Y, \pi^*(\E), \pi^*(H))$ in the
absolute  twisted cohomology $H^\bullet(Y,  \pi^*(\E), \pi^*(H))$ is zero. We argue this as follows. 
For $s\in [0,1]$, let $\iota_s : Y \to Y$ be defined as $\iota_s(x, t) = (x, st)$. Then $\iota_s$ is a smooth map such that $\iota_0(x, t) = (x, 0)$, $\iota_1(x, t)=(x, t)$, and $\iota_s(x, 0)=(x,0)$. By the homotopy invariance 
of the twisted absolute cohomology, Proposition \ref{prop:homotopy}, $\iota_0^* = \iota_1^*= {\rm Id} :  
H^\bullet(Y,  \pi^*(\E), \pi^*(H)) \to H^\bullet(Y,  \pi^*(\E), \pi^*(H))$. On the other hand, one has
$0 = \iota_0^*\big|_{{\rm Image}(H^\bullet(Y, \partial Y, \pi^*(\E), \pi^*(H)))}$. Therefore
\begin{equation}\label{eqn:sig2}
{\rm Sign}(X \times [0,1], \pi^*(\E), \pi^*(H))=0.
\end{equation}
The same proof shows when $H=0$ that (see also equation (2.3) of \cite{APS2}) 
\begin{equation}\label{eqn:sig3}
\eta(D^\E, g_1, \lambda_1)
-\eta(D^\E, g_0, \lambda_0) =  \Rank(\E)\int_Y\alpha_0(x)dx
\end{equation}
Combing {{\eqref{eqn:sig1}, \eqref{eqn:sig2} and  \eqref{eqn:sig3}}}, the Theorem follows.
\end{proof}

\subsection{Variation of the flux in a
cohomology class}\label{sect:var-H}
We continue to assume that $\dim X$ is odd and use the same notation as above. {{In the sequel, a closed odd degree form $\omega$  is said to be of the ``form above'' if it has the decomposition 
$$
\omega=\omega_0 + i \omega_1, \quad \text{ with } \omega_0\in \Omega^{4\bullet + 3}(X, \RR) \text{ and } \omega_1\in \Omega^{4\bullet +1}(X,\RR).
$$
so $\omega=\sum_j i^{j+1} \omega_{2j+1}$.}}
Then we have,

\begin{theorem}[flux representative independence of the rho invariant]
\label{thm:indept-H}
Let $X$ be a compact, oriented Riemannian manifold of odd dimension and $\E$,
a flat hermitian vector bundle over $X$.
Suppose $H_0$ and $H_1$ are closed differential forms on $X$ of odd degrees
representing the same de Rham cohomology class, and are of the form above. 
Then one has the equality of rho invariants 
$\rho(X,\E,H_1)=\rho(X,\E,H_0)$ and eta invariants $\eta(D^\E_{H_1})
=\eta(D^\E_{H_0})$.
\end{theorem} 

\begin{proof}\
{{Let $B$ be an even degree
form so that $H_1=H_0+dB$. Notice that we can assume, subtracting from $B$ a closed even degree form if necessary, that $B$ is of the form $B_0+iB_1$ with $B_0\in \Omega^{4\bullet +2}(X,\RR)$ and $B_1\in \Omega^{4\bullet}(X,\RR)$.}}

Consider the manifold with boundary $Y = X \times [0, 1]$, where the boundary $\partial Y 
= X \times \{0\} - X \times \{1\}$. 
Define the odd degree form on $Y$,
$$
H_t  =  (1-f(t))H_0 + f(t) H_1,  \qquad t\in [\epsilon, 1-\epsilon];
$$
where $f: [0, 1] \to  [0, 1]$
be a smooth function such that $f (t)= 0$ for all $t$ small and $f (t)= 1$ for all $1-t$ small.

Then $\,H =H_t - \frac{\partial}{\partial t} f(t)B \wedge dt\,$ defines a closed differential {{odd}} degree form on $Y$ which is of the form above on $Y$. By fiat, $H$ is boundary compatible.

\begin{figure}[h]
\includegraphics[height=3in]{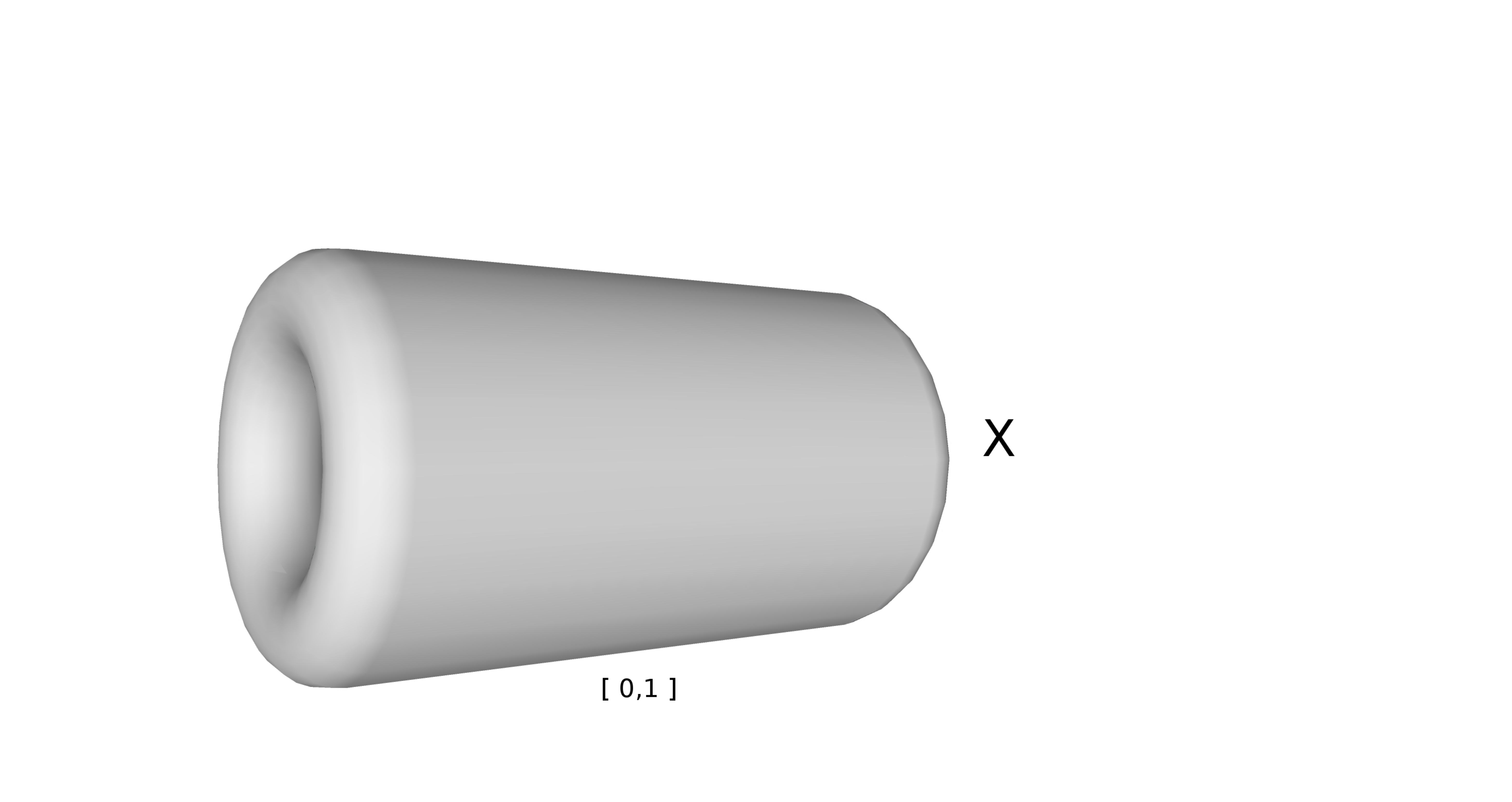}\\
\caption{Cylindrical manifold, $\quad H= H_t - \frac{\partial}{\partial t} f(t) B \wedge dt$}
\label{fig:box-metric}
\end{figure}

Consider the product hermitian metric on $\pi^*(\E)$ (where $\pi : Y \to X$
denotes projection onto the first factor) and the product Riemannian metric on $Y$. 
Applying the signature theorem for the twisted signature complex, Corollary \ref{cor:sig}, we get
\begin{equation}\label{eqn:sig1a}
{\rm Sign}(X \times [0,1], \pi^*(\E), H)=\Rank(\E) \int_Y \alpha_0(x)dx - \left(\eta(D^\E_{H_1})
-\eta(D^\E_{H_0})\right)\\
\end{equation}
We next {{show}} that ${\rm Sign}(X \times [0,1], \pi^*(\E), H)=0$. This follows as in the proof of Theorem \ref{thm:indept} from the fact that
the image of the relative twisted cohomology $H^\bullet(Y, \partial Y, \pi^*(\E), H)$ in the
absolute  twisted cohomology $H^\bullet(Y,  \pi^*(\E), H)$ is zero. We argue again this using $\iota_s : Y \to Y$  defined as $\iota_s(x, t) = (x, st)$. By the 
independence of the twisted absolute and relative cohomology on the flux form $H$, Lemma \ref{lem:fluxindep}, 
$\iota_0^* = \iota_1^*= {\rm Id} :  
H^\bullet(Y,  \pi^*(\E), H) \to H^\bullet(Y,  \pi^*(\E), H)$. On the other hand, one has
$0 = \iota_0^*\big|_{{\rm Image}(H^\bullet(Y, \partial Y, \pi^*(\E), \pi^*(H)))}$. Therefore
\begin{equation}\label{eqn:sig2a}
{\rm Sign}(X \times [0,1], \pi^*(\E), H)=0.
\end{equation}
On the other hand, it has been shown in equation (2.3) of \cite{APS2} that
\begin{equation}\label{eqn:sig3a}
0=\eta(D^\E)
-\eta(D^\E) =  \Rank(\E)\int_Y \alpha_0(x)dx.
\end{equation}
Combing \eqref{eqn:sig1a}, \eqref{eqn:sig2a} and  \eqref{eqn:sig3a}, the theorem follows.
\end{proof}

\begin{corollary}
Let $X$ be a compact, oriented Riemannian manifold of odd dimension and $\E$,
a flat hermitian vector bundle over $X$.
Suppose $H$ is a closed differential form on $X$ of odd degree and of the form above.
Then the twisted rho invariant
$\rho(X,\E,H)$ is a differential invariant of the triple $(X, \E, [H])$, {{i.e. given another smooth Riemannian manifold of odd dimension and a diffeomorphism $f:X'\to X$, the following equality holds
$$
\rho (X, \E, H) = \rho (X', f^*\E, f^*H),
$$
when $f^*\E$ is induced with the pulled-back flat connection.}}
\end{corollary}

\begin{proof}\
{{We have proven that $\rho (X', f^*\E, f^*H)$ does neither depend on the chosen hermitian structure of $f^*\E$ nor on the given riemannian metric of $X'$. Therefore, we can compute it by  pulling-back under $f$ the hermitian structure of $\E$ and the riemannian metric of $X$. But the twisted signature operator $D^{f^*\E}_{f^*H}$ obtained on $X'$ using these pulled-back structures, is then conjugated with $D^\E_H$. Henceforth, the proof is complete.
}}
\end{proof}

\subsection{Functorial properties of rho invariant}\label{sect:func}

Here we state and prove the basic functorial properties of the rho invariant
 for the twisted signature complex.

\begin{proposition}
Let $X$ be a compact, oriented Riemannian manifold of odd dimension and $\E_1,\E_2$,
flat Hermitian or orthogonal vector bundles on $X$.
Suppose $H$ is a closed odd-degree form on $X$ {{of the above form}}.
Then 
$$
\eta(D^{\E_1\oplus\E_2}_H) = \eta(D^{\E_1}_H) + \eta(D^{\E_2}_H)\text{ and }\rho(X,\E_1\oplus\E_2,H)=\rho(X,\E_1,H)+ \rho(X,\E_2,H).
$$
\end{proposition}

\begin{proof}
On $\Omega^\bullet(X,\E_1\oplus\E_2)\cong\Omega^\bullet(X,\E_1)\oplus
\Omega^\bullet(X,\E_2)$, the odd signature
 operator $D^{\E_1\oplus\E_2}_H=
D^{\E_1}_H\oplus D^{\E_2}_H$ is block-diagonal.
The proposition follows.
\end{proof}

\begin{proposition}
Let $X_1$, $X_2$ be two compact oriented Riemannian manifolds of odd dimension with the same universal
covering manifold.
Suppose the fundamental group $\pi_1(X_1)$ is a subgroup of $\pi_1(X_2)$.
Let $\alpha_1$ be a representation of $\pi_1(X_1)$ and let $\alpha_2$ be the
induced representation of $\pi_1(X_2)$.
Denote the flat vector bundles associated with $\alpha_1$, $\alpha_2$ by
$\E_1$, $\E_2$, respectively.
Suppose the closed odd-degree forms $H_1$ on $X_1$ and $H_2$ on $X_2$
pull-back to the same form on the universal covering.
Then 
$$
\eta(D^{\E_1}_{H_1}) = \eta(D^{\E_2}_{H_2}), \qquad \rho(X_1,\E_1,H_1)=\rho(X_2,\E_2,H_2).
$$
\end{proposition}

\begin{proof}
By assumption, there is a fibre bundle $p: X_1 \to X_2$ with fibre the homogeneous set $\pi_1(X_2)/\pi_1(X_1)$.
A Riemannian metric on $X_2$ induces one on $X_1$ and since $\alpha_2$ is induced from $\alpha_1$, then 
the canonical isomorphism 
$\Omega^\bullet(X_1,\E_1,H_1)\cong\Omega^\bullet(X_2,\E_2,H_2)$
is an isometry as explained in Theorem 2.6 in \cite{RS} and therefore we deduce that the
spectrums of the twisted odd signature operators
$D^{\E_1}_{H_1}$ and 
$D^{\E_2}_{H_2}$  coincide. The proposition follows.
\end{proof}

\begin{proposition}
In addition to the conditions of Proposition~\ref{prop:ku}, assume that both
$X_1$ and $X_2$ are closed, oriented Riemannian manifolds, dimension of $X_1$ is odd and 
dimension of $X_2$ is even.
Then we have
\begin{align*}
\eta(D^{\E_1\boxtimes\E_2}_{H_1\boxplus H_2}) & = 
\eta(D^{\E_1}_{H_1})\, {\rm Sign}(X_2,\E_2),\\
\rho (X_1\times X_2,\E_1\boxtimes\E_2,H_1\boxplus H_2) & =
\rho(X_1,\E_1,H_1)\, {\rm Sign}(X_2,\E_2).
\end{align*}
\end{proposition}

\begin{proof} Upon identifying 
$$
\Omega^\pm(X_1\times X_2,\E_1\boxtimes\E_2) \cong 
\Omega^{even}(X_1,\E_1)\boxtimes \Omega^\pm(X_2,\E_2),
$$
we can write the operator 
$$
D^{\E_1\boxtimes\E_2}_{H_1\boxplus H_2} = {D^{\E_1}_{H_1} \otimes 1\qquad
1\otimes B^{\E_2}_{H_2}
\choose\quad\; 1\otimes
{B^{\E_2}_{H_2}}^\dagger
 \qquad \;-D^{\E_1}_{H_1} \otimes 1}.
$$
By the discussion following Theorem 4.5 in \cite{APS3}, we conclude that
$$
\eta(D^{\E_1\boxtimes\E_2}_{H_1\boxplus H_2})  = 
\eta(D^{\E_1}_{H_1})\, {\rm Sign}(X_2,\E_2),
$$
and Proposition \ref{prop:ku} the proposition follows.
\end{proof}

\subsection{Relation to generalized geometry}

Recall that in generalized geometry \cite{Hi,Gua1}, the bundle
$TX\oplus T^*X$ has a bilinear form of signature $(n,n)$ given by
$$
\langle\xi_1+W_1,\xi_2+W_2\rangle:=(\xi_1(W_2)+\xi_2(W_1))/2,
$$
where for $i=1,2$, $\xi_i$ are $1$-forms and $W_i$ are vector fields on $X$. 
A {\em generalized metric} on $X$ is a reduction of the structure group
$O(n,n)$ to $O(n)\times O(n)$.
Equivalently, a generalized metric is a splitting of $TX\oplus T^*X$ to
a direct sum of two sub-bundles of rank $n$ so that the bilinear form is
positive on one and negative on the other.
The positive sub-bundle is the graph of $g+B\in\Gamma(\Hom(TX,T^*X))$, where
$g=g_X$ is a usual Riemannian metric on $X$ and $B$ is a $2$-form on $X$. 

As argued in \cite{MW}, we can conclude that deformation of $H$ by a $B$-field is equivalent
to deforming the usual metric to a generalized metric.
In this way, deformations of the usual metric and that of the flux by a
$B$-field are unified to a deformation of generalized metric. 
Theorems~\ref{thm:indept} and \ref{thm:indept-H} state that the torsion is
invariant under such a deformation.

\section{Spectral flow and calculations of the twisted eta invariant}
\label{sect:calc}

We want to get the formula for the dependence of the eta invariant $\eta(D^\E_H) $ on $[H]$.
We employ the method of 
spectral flow for a path of self-adjoint elliptic operators, which is generically the net number of eigenvalues
that cross zero, which was first defined by \cite{APS3}, to get a formula for the difference 
$\eta(D^\E_H) - \eta(D^\E)$. 

Let $(X, g)$ be an oriented Riemannian manifold. 
Define $S^{-H} \in \Omega^1(X, {\mathcal S}{\mathcal O}(TM))$, which is a degree one form with values in 
the skew-symmetric endomorphisms of $TM$, as follows. For $\alpha, \beta, \gamma \in TM$, set 
$$
g(S^{-H} (\alpha)\beta, \gamma) = - 2H(\alpha, \beta, \gamma).
$$
Let $\nabla^L$ denote the Levi-Civita connection of $(X, g)$ and $\nabla^{-H} = \nabla^L + S^{-H} $ 
be the Riemannian connection whose curvature is denoted by $\Omega_X^{-H}$.

\begin{proposition}\label{prop:getzler}
Consider the smooth path of self-adjoint elliptic operators $\{D^\E_{uH}, \, : \, u\in [0,1]\}$
on a $2\ell +1$ dimensional closed, oriented, Riemannian manifold, where $H$ is closed and degree $H$ is equal to $3$. Then
$$
{\rm sf}(D^\E, D^\E_H) = \frac{1}{(-2\pi i)^{\ell+1}} \int_X H \wedge \LL(\Omega_X^{-H}) + \frac{1}{2}(\eta(D^\E_H) - \eta(D^\E))
$$
where $\LL(\cdot)$ denotes the Atiyah-Hirzebruch characteristic polynomial, $\Omega_X^{-H}$ denotes the curvature of the 
Riemannian connection $\nabla^{-H}$, ${\rm sf}(D^\E, D^\E_H)$ denotes the spectral flow of the given path.
\end{proposition}

\begin{proof} By theorem 2.6 \cite{Getzler}, 
$$
{\rm sf}(\widetilde D^\E, \widetilde D^\E_H) = \left(\frac{\epsilon}{\pi}\right)^{1/2}\int_0^1 \Tr\left([H \wedge \star - \star H \wedge]
e^{-\epsilon {D^\E_{uH}}^2}\right) du
 + \frac{1}{2}(\eta(D^\E_H) - \eta(D^\E))
$$
where $\widetilde D^\E_H$ denotes the operator $ D^\E_H + P_{\E, H}$ where $P_{\E, H}$ denotes the 
orthogonal projection onto the nullspace of $ D^\E_H$. Following the idea of the proof of Theorem 2.8 in \cite{Getzler}
the technique of \cite{BGV}, and the Local Index Theorem 1.7 of Bismut \cite{Bismut}, we may make the following replacements:
\begin{align*}
\epsilon^{1/2} [H \wedge \star - \star H \wedge]\qquad & \text{by} \qquad H\wedge\\  
e^{- (\epsilon^{1/2}D^\E_{uH})^2} \qquad & \text{by} \qquad  \LL(\Omega_X^{-H}) \wedge e^{(d+uH)^2}\\
\Tr(\cdot) \qquad & \text{by} \qquad \frac{-i\pi^{1/2}}{(-2\pi i)^{\ell+1}} \displaystyle\int_X \tr(\cdot)
\end{align*}  
This then enables us to replace $ \left(\frac{\epsilon}{\pi}\right)^{1/2}\displaystyle\int_0^1 \Tr\left([H \wedge \star - \star H \wedge]
e^{-\epsilon {D^\E_{uH}}^2}\right) du$ by 
$$
 \frac{1}{(-2\pi i)^{\ell+1}} \int_X H \wedge \LL(\Omega_X^{-H}),
$$
and the proposition follows.
\end{proof}

As a special case of  Proposition \ref{prop:getzler}, one obtains the following calculation,

\begin{corollary}
Let $X$ is a compact oriented manifold of dimension 3 and
$\E$ is a flat vector bundle associated to an orthogonal or unitary
representation of $\pi_1(X)$.
Let $H$ be a closed $3$-form on $X$.
Consider the smooth path of self-adjoint elliptic operators $\{D^\E_{uH}, \, : \, u\in [0,1]\}$. Then
$$
\eta(D^\E_H) - \eta(D^\E) = 2\, {\rm sf}(D^\E, D^\E_H)  + \frac{h}{2\pi^2}
$$
where $h=[H] \in \mathbb R\cong H^3(X, \RE)$.
\end{corollary}

In the setting of the Corollary above, one knows that there is nontrivial spectral flow
as the twisted cohomology jumps at $u=0$. However, away from $u=0$, there is no spectral 
flow, and in this region of the path, the eta invariant $\eta(D^\E_{uH})$ is a smooth affine linear
function.

\section{Homotopy invariance}\label{sect:homotopy}

Our goal in this last section is to prove the homotopy invariance of our twisted rho invariant. Given a smooth closed oriented manifold $M$ of dimension $2m-1$, we denote  by $\Gamma$ the fundamental group of $M$ (with respect to some fixed base point) and we consider the  Mishchenko bundle $\maM:= \tM\times_\Gamma C^*\Gamma$. Here $C^*\Gamma$ is the maximal $C^*$-algebra of the group $\Gamma$. The twisted signature operator $D_H$ then gives rise to a regular operator $\maD_H$ on the $C^*\Gamma$ Hilbert module $L^2(M, \maM\otimes \Lambda^{even}T^*M)$.  

 If we are now given   a unitary representation $\sigma:\Gamma\to U(N)$ then using our twisted by $H$ signature operator on $M$ and the flat bundle $\E=\E_\sigma$ over $M$ associated with  $\sigma$, we may form as in Definition \ref{DefRho} the rho invariant $\rho (M, \E, H)$. Notice that the operator $D_H^\E$ can be recovered from $\maD_H$ as $D_H^\E=\maD_H\otimes_{C^*\Gamma} Id$ through the classical  identification \cite{BP}.
 In the same way, given a smooth map $h: M'\to M$ from a second closed oriented manifold $M'$ to $M$, we can pull-back the flat hermitian vector bundle $\E$ as well as the differential form $H$ and define the twisted rho invariant $\rho (M', h^*\E, h^*H)$.  Recall on the other hand the maximal Baum-Connes assembly map $\mu_{\Gamma, \rm{max}}$ associated with $\Gamma$ \cite{BaumConnes}.  We prove here the following 
 
\begin{theorem}\label{HomotopyInvariance}
Assume that $\Gamma$ is torsion free and that the maximal Baum-Connes map $\mu_{\Gamma, \rm{max}}$ is an isomorphism. If there exists an oriented  homotopy  equivalence $h:M'\to M$ between the smooth closed oriented $2m-1$ dimensional manifolds $M$ and $M'$, then
$$
\rho (M', h^*\E, h^*H) = \rho (M, \E, H).
$$
\end{theorem}

When $H=0$ we recover a now classical theorem, see for instance \cite{Keswani}. The proof we have adopted here follows the new method introduced by Higson and Roe in \cite{HigsonRoeRho} which in turn simplifies  Keswani's original proof \cite{Keswani}. Since the constructions are easy generalizations of this proof, we shall be brief. Recall that we are only interested in  our twisted signature operator $D_H^\E$ acting of {\em{even} } forms. Therefore, it is easy to check that  our operator coincides with the signature operator considered in \cite{HigsonRoe1}. We may then write $\maD_H$ as well as $D_H^\E$ as
$$
\maD_H=i\maB_H \maJ \text{ and } D_H^\E = i B^\E_H J\text{ with } \maJ \text{ induced by } J=i^{p(p-1)+m-1} \star \text{ on } \Omega^p (M,\E).
$$
Then $\maJ$ is a self-adjoint involution, and  $\maB_H \pm \maJ$ are invertible operators as regular operators on the Hilbert module $L^2(M, \maM\otimes \Lambda^{even}T^*M)$ \cite{Lance}. Moreover, we have: 
$$
(\maD_H - i)(\maD_H+i)^{-1} = (\maB_H - \maJ)(\maB_H+\maJ)^{-1}.
$$
The same relation holds for the Hilbert space operator $D_H^\E$. 

The pull-back map $h^*$ on smooth forms does not - in general - induce a closable operator between the Hilbert spaces of $L^2$ forms and hence no more between the Michschenko Hilbert modules. We denote by $\Delta_H$ the Laplacian operator and we introduce  the operator $\maA_H:=e^{-\Delta_H}  h^* e^{-\Delta'_{h^*H}}$ which then extends to an adjointable operator, still denoted 
$$
\maA_H: L^2(M', \maM'\otimes \Lambda^{even}T^*M') \longrightarrow L^2(M, \maM\otimes \Lambda^{even}T^*M).
$$
Moreover, with our previous notations and extending the twisted de Rham operators to the Hilbert modules, we have  $\maA_H d'_{h^*H} = d_H \maA_H$. Using the representation $\sigma$, we also define the corresponding bounded operator $A_H^\E$ and we have the similar relation  $A_H^\E {\nabla'}_{h^*H}^{h^*\E} = \nabla_H^\E A_H^\E$. 
Using twisted Hodge theory \cite{MW}, it is easy to see that exactly as in the untwisted case, the operator $\maA_H$ induces an isomorphism between the twisted Hilbert module de Rham cohomologies. 

Following \cite{Keswani} (see also \cite{HigsonRoeRho}), we introduce a  path $(\maJ_{t, H})_{0\leq t\leq 2}$ of adjointable operators on the direct sum Hilbert module 
$$
L^2(M, \maM\otimes \Lambda^{even}T^*M) \oplus L^2(M', \maM'\otimes \Lambda^{even}T^*M')
$$ 
For   $0\leq t\leq 1/2$, we set
$$
\maJ_{t, H} = \left(\begin{array}{cc} \maJ & 0 \\ 0 & -\maJ ' - 2t (\maA_H^* \maJ \maA_H - \maJ ')
\end{array}\right);
$$
For $1/2\leq t \leq 1$, we set
$$
\maJ_{t, H} := \left(\begin{array}{cc} \sin (\pi t) \maJ & \cos (\pi t) \maJ \maA_H \\ \cos (\pi t) \maA_H^* \maJ & - \sin (\pi t) \maA_H^* \maJ \maA_H
\end{array}\right);
$$
and lastly for $1\leq t\leq 2$, we set
$$
\maJ_{t, H} :=   \left(\begin{array}{cc} 0 & \exp (\pi i t) \maJ \maA_H \\ \exp (\pi i t) \maA_H^* \maJ & 0
\end{array}\right).
$$
We also denote for $0\leq t \leq 1$ by $\maT_{t, H}$ the operator defined on $p$-forms by the equality
$$
\maJ_{t, H} = i ^{p(p-1)+m-1} \maT_{t, H}.
$$

\begin{lemma}
\begin{itemize}
\item The operator $\maT_{t, H}$ is self-adjoint. 
\item If ${\hat d}_H=d_H\oplus d'_{h^*H}$ and $H$ is pure imaginary, then the following formula holds on $p$-forms $
 \maT_{t, H} {\hat d}_H + (-1)^p ({\hat d}_H)^*  \maT_{t, H} = 0.$
\item Under the assumptions of the previous item, the operator induced by $\maT_{t,H}$ is an isomorphism from the homology of ${\hat d}_H$ to the homology of $({\hat d}_H)^*$.
\end{itemize}
\end{lemma}

\begin{proof}
\begin{itemize}
\item Since we are working with odd dimensional manifolds, the star operators $\star$ and $\star'$ are self-adjoint. Therefore, it is easy to check that $\maT_{t, H} $ is then self-adjoint.
\item Again because we are working on odd dimensional manifolds, the following relations hold on $p$-forms:
$$
 (d_H)^* \star + (-1)^p \star d_{-{\bar H}} = 0 \text{ and }  (d'_{h^*H})^* \star ' + (-1)^p \star ' d'_{-{h^*{\bar H}}} = 0.
$$
Now, since $H$ is pure imaginary, so is $h^*H$ and a long but straightforward  inspection yields the announced formula.
\item We recall  that the operator $\maA_H$ induces an isomorphism from the twisted by $h^*H$ Hilbert module de Rham cohomology of $M'$ to the twisted by $H$ Hilbert module de Rham cohomology of $M$. 

Using again the twisted Hodge decomposition theorem \cite{MW}, this allows to deduce that $\maT_{t, H}$ induces an isomorphism between the twisted homologies. Indeed, one first works with smooth forms and then the exponentials of the twisted laplacians induce the identity operators on homologies.  The rest of the proof is classical.
\end{itemize}
\end{proof}

We may and will assume from now on that $H$ is pure imaginary. Recall that $\maB_H=d_H + (d_H)^*$ and similarly for $\maB'_{h^*H}$.

\begin{lemma}
Set ${\hat\maB}_H:= \maB_H\oplus \maB'_{h^* H}$. Then for  $0\leq t \leq 2$, the operators ${\hat\maB}_H \pm \maJ_{t, H}$ are invertible as regular operators on the Hilbert module 
$$
L^2(M, \maM\otimes \Lambda^{even}T^*M) \oplus L^2(M', \maM'\otimes \Lambda^{even}T^*M').
$$
\end{lemma}

\begin{proof}
The proof given in \cite{HigsonRoe1}, with the small modifications introduced in \cite{BR} to encompass regular operators, works in our situation. We recall it for the convenience of the reader. We consider the mapping cone complex of the chain map $\maJ_{t,H}$ from the complex defined by $-{\hat d}_H$ to the complex defined by $({\hat d}_H)^*$ whose differential is given by $\left(\begin{array}{cc} ({\hat d}_H)^* & 0 \\ \maJ_{t,H} &   {\hat d}_H  \end{array}\right)$. Using the previous lemma we know that $\maJ_{t,H}$ indices an isomorphism in cohomology and hence this mapping cone complex is acyclic. It is then proved in \cite{BR} that this implies that the operator 
$$
\left(\begin{array}{cc} ({\hat d}_H)^* + {\hat d}_H & \maJ_{t,H}^* \\ \maJ_{t,H} &   {\hat d}_H + ({\hat d}_{H})^* \end{array}\right),
$$
is invertible. Since $\maJ_{t,H}$ is self-adjoint for any $t$ and restricting this invertible operator to the the diagonal  of $L^2(M, \maM\otimes \Lambda^{even}T^*M) \oplus L^2(M', \maM'\otimes \Lambda^{even}T^*M')$, and which is obviously preserved, we deduce that ${\hat\maB}_H + \maJ_{t, H}$ is invertible. The same argument works for ${\hat\maB}_H - \maJ_{t, H}$ by using the anti-diagonal.
\end{proof}

Recall that $\Gamma$ is assumed to be torsion free. Following the general construction in \cite{HigsonRoe2} and the corresponding one for the group $\Gamma$ studied in \cite{HigsonRoeRho}, we get an analytic surgery group $\maS_1\Gamma$ which fits into the long exact sequence
$$
\cdots \to K_0(B\Gamma) \rightarrow K_0(C^*\Gamma) \rightarrow \maS_1\Gamma \rightarrow K_1(B\Gamma) \rightarrow K_1(C^*\Gamma)\to \cdots 
$$
where the map $K_*(B\Gamma) \rightarrow K_*(C^*\Gamma)$ is the Baum-Connes map. The main result of \cite{HigsonRoeRho} which is used here is the construction of the rho morphism associated with $\sigma$:
$$
\maS_1\Gamma \longrightarrow \R.
$$
This is constructed as follows. Associated with our $N$-dimensional representation $\sigma$, one associates in a similar way a group $\maS_1(\sigma)$ obtained by using the pair of representations $(\sigma, \ep_N)$ where $\ep_N$ is the trivial representation and then obtains a ``composition morphism'' 
$$
\maS_1\Gamma \longrightarrow \maS_1(\sigma).
$$
The group $\maS_1(\sigma)$ is actually isomorphic to a geometric group $\maS_1^{geom}(\sigma)$ for which the definition of the rho morphism is clearer and whose cycles we briefly recall now. A geometric cycle  is given by a quintuple $(X, E, f, D, n)$ where $(X,E,f)$ is an odd geometric Baum-Douglas $K$-cycle for $B\Gamma$, $D$ is a specific choice of Dirac operator and $n$ is an integer. What is important in the choice of $D$ is that the commutator $[D, \varphi]$ coincides with the usual commutator with the Dirac operator, hence is Clifford multiplication by the gradient of the function $\varphi$. The equivalence relations mimic the Baum-Douglas ones, but do involve the representation $\sigma$ and the APS index map where the integer $n$ is also involved. Higson and Roe used the rho invariant to construct the morphism 
$$
\rho: \maS^{geom}_1(\sigma) \longrightarrow \R.
$$ 
More precisely, let us denote by $\xi $ the boundary contribution in the APS theorem:
$$
\xi := (\eta - \dim\Ker )/ 2.
$$
Then
given $ (X, E, f,  D, n)$, as above, we may use $f$ and $\sigma$ to build up the twisted operator $D_{f,\sigma}$ and moreover tensor it by $E$ to get a self-adjoint operator $D_{f,\sigma}^E$ on $X$ whose $\xi$ invariant $\xi (D_{f,\sigma}^E)$ is well defined. Doing so for the trivial $1$-dimensional representation we get in the same way the invariant $\xi (D_f^E)$. We then set 
$$
\rho_\sigma (X, E, f,  D, n) : = \xi  (D_{f,\sigma}^E) - \dim (\sigma) \xi (D_f^E) + n.
$$
This construction induces precisely the Higson-Roe morphism $\rho: \maS^{geom}_1(\sigma) \rightarrow \R$ announced above.
We thus end up with the the composite map
$$
 \maS_1 \Gamma \rightarrow \maS_1(\sigma) \stackrel{\cong}{\leftarrow} \maS^{geom}_1(\sigma) \stackrel{\rho}{\longrightarrow} \R.
$$

Notice that the twisted signature operator is a specific choice of Dirac operator  since the commutator of the twist $H$ with any function is trivial. Hence, it is clear from the very definition  that we get a cycle which represents a class in $\maS_1^{geom}(\sigma)$ and which is given by 
$$
(M, E, f, D_H, 0) \amalg -(M', E', f\circ h, D'_{h^*H}, 0),
$$
where $E$ and $E'$ and the usual grassmannian bundles defining the (twisted) signature operators $D_H$ and $D'_{h^*H}$,  and $f:M\to B\Gamma$ is the classifying map for the universal cover of $M$. Here of course $\Gamma$ is the fundamental group of $M$ and hence also of $M'$.

We are now in position to finish the proof of the theorem.

{\em{End of the proof of Theorem \ref{HomotopyInvariance}:}} The rho morphism defined above differs from ours by the kernel dimension. Since this defect is a twisted cohomology dimension, we know that it is homotopy invariant \cite{MW}. Therefore, we can work with the rho morphism as well and imitate the Higson-Roe proof to deduce the homotopy invariance of our twisted by $H$ $\rho$ invariant.  
Hence,  it  remains to show that the image under the group morphism $\maS_1(\sigma) \longrightarrow \R$ of the class  $[h, H]$ defined above is precisely the difference
$$
\xi (M, \E, H) - \xi (M', h^*\E, h^*H).
$$
But this is again a  straightforward generalization of the argument given in \cite{HigsonRoeRho} (see pages 43-45).  Indeed, the image in  $\maS_1(\sigma)$ of the geometric class represented by  $
(M, E, f, D_H, 0) \amalg -(M', E', f\circ h, D'_{h^*H}, 0),
$ can be described using a Bott argument (see \cite{HigsonRoeRho}[Lemma 7.4]) and is precisely represented by   $[h, H]:= (P_H^\E, \alpha_{t,H})$, where  the operator $P_H^\E$ is $P_H^\E := \frac{1}{2} (\varphi ({\hat D}_H^\E) + Id)$ with $\varphi $ being the chopping function $\frac{2}{\pi} \arctan$, and the path $\alpha_{t, H}$ being given by
$$
\alpha_{t,H} := ({\hat \maB}_H - \maJ_{t, H})({\hat \maB}_H + \maJ_{t, H})^{-1}, \text{ for }0\leq t \leq 1
$$
  and 
  $$
    \alpha_{t,H} := ({\hat \maB}_H - \maJ_{1, H})({\hat \maB}_H + \maJ_{t, H})^{-1} \text{ for } 1\leq t\leq 2.
$$

\bigskip

\appendix
\section{Principle of not feeling the boundary}
In this appendix, we give details of what was claimed about the integrand in Proposition \ref{prop:sig1}
in Remark \ref{rem:sig1}.

We shall assume 
that ${DY}$ is the closed manifold which is the double of the manifold with boundary $Y$. $DY$
comes with a smooth Riemannian metric (since the metric on 
$Y$ is assumed to be a product metric near $\partial Y=X$), and the flat bundle $\E$ on $Y$ also extends to $D\E$ on $DY$. 
The adapted odd degree closed form $H$ on $Y$ induces a closed odd degree form $DH$
on $DY$, and let $B_{DH}^{D\E}$ denote the twisted signature operator on $DY$.
The key technical part of our paper is to use
a quantitative version of the following \\

\noindent{\bf Principle of not feeling the boundary.} {\em
Let $k_{+}(t,x,y)$ denote the heat kernel of $(B_{DH}^{D\E})^\dagger B_{DH}^{D\E}$
and $k_{-}(t,x,y)$ denote the heat kernel of $B_{DH}^{D\E}(B_{DH}^{D\E})^\dagger $
on $DY$.
Let $p_+(t,x,y)$ denote the heat kernel of  $(B_{H}^\E)^\dagger B_{H}^\E$ 
and $p_{-}(t,x,y)$ denote the heat kernel of $B_{H}^\E(B_{H}^\E)^\dagger$
which is associated to the Atiyah-Patodi-Singer global boundary conditions.
Then for $x, y \not\in \partial Y=X$ and as $t\downarrow 0$},
$$
k_\pm(t,x,y) \sim  p_\pm(t,x,y).
$$

This is a generalisation of the well known principle on functions due 
to M. Kac \cite{K}. 
For Laplacians acting on bundles and for relative or more generally for local boundary conditions, this is 
contained in \cite{DM}.
The result below is not optimal, but is sufficient for our purposes.
It follows from Theorem 1.11.7 in \cite{GilkeyBook} 2nd edition, and equation (3.4) in \cite{APS1} that

\begin{theorem}
In the notation above, for $x \in Y$ and $x\not\in \partial Y=X$ and as $t\downarrow 0$,
$$
\tr(k_\pm(t,x,x)) = \tr( p_\pm(t,x,x)) + O(t^{3/8}).
$$
In particular, the coefficient of $t^0$ in the small time asymptotic expansions of $\tr(k_\pm(t,x,x))$ and of $ \tr( p_\pm(t,x,x))$
coincide, for $x \in Y$ and $x\not\in \partial Y=X$.
\end{theorem}

Therefore it suffices to analyse the  small time asymptotic expansions of $ \tr( k_\pm(t,x,x))$.

Let $\widetilde k_\pm(t,x,y)$ denote the heat kernel of the operator  $\widetilde{(B_{DH})^\dagger}\widetilde B_{DH}$
and $\widetilde k_{-}(t,x,y)$ denote the heat kernel of $\widetilde{B_{DH}}\widetilde{(B_{DH})^\dagger }$
on the universal covering space $\widetilde{DY}$ of the closed manifold $DY$.
Then one has

\begin{theorem}
In the notation above,  for $x\in \widetilde{DY}$ and as $t\downarrow 0$
$$
|\tr(k_\pm(t,\bar x,\bar x)) - \tr(\widetilde k_\pm(t,x,x))\Rank(\E)|
  \le C_1t^{-n/2}\sum_{\gamma\in\Pi\setminus\{1\}}
  \exp\Big[-C_2\Big(\frac{d(x,x\gamma)}{t}\Big)^{2}\Big].   
$$
where $\bar x\in DY$ stands for the projection of $x\in\widetilde{DY}$, and $\Pi$ denotes the fundamental group of $DY$.
\end{theorem}

\begin{proof}
By the Selberg principle, one has for $x,y\in\widetilde{DY}$,
\[ k_{\pm}(t,\bar x,\bar y)=
   \sum_{\gamma\in\Pi}\widetilde k_\pm(t,x,y\gamma)\rho(\gamma),   
   \]
where $\bar x\in DY$ stands for the projection of $x\in\widetilde{DY}$, $\Pi$ denotes the fundamental group of $DY$
and $D\E$ is determined by the unitary representation $\rho$ of $\Pi$.
It follows that 
\[  \tr(k_{\pm}(t,\bar x,\bar x)-\tr(\widetilde k_\pm(t,x,x)) \Rank(\E)
    =\sum_{\gamma\in\Pi\setminus\{1\}}\tr(\widetilde k_\pm(t,x,x\gamma)\rho(\gamma)), 
      \]
Since $\rho$ is a unitary representation, and using the matrix estimate $|\tr(AB)| \le ||B||  \tr(|A|)$, one has
\[ |\tr(k_{\pm}(t,\bar x,\bar x)-\tr(\widetilde k_\pm(t,x,x)) \Rank(\E)|\le
   \sum_{\gamma\in\Pi\setminus\{1\}} \tr(|\widetilde k_\pm(t,x,x\gamma)|).  \]
The off-diagonal Gaussian estimate for the heat kernel on $\widetilde{DY}$ is
\cite{BrSu}
\[  |\widetilde k_\pm(t,x,y)|\le C_1t^{-n/2}
       \exp\Big[-C_2\Big(\frac{d(x,y)}{t}\Big)^{2}\Big],   \]
where $d(x,y)$ is the Riemannian distance between $x,y\in\widetilde{DY}$.
Therefore 
\[  |\tr(k_{\pm}(t,\bar x,\bar x)-\tr(\widetilde k_\pm(t,x,x)) \Rank(\E)|
  \le C_1\Rank(\E)\, t^{-n/2}\sum_{\gamma\in\Pi\setminus\{1\}}
  \exp\Big[-C_2\Big(\frac{d(x,x\gamma)}{t}\Big)^{2}\Big].   \]
By Milnor's theorem \cite{Mi}, there is a positive constant $C_3$ such that 
$d(x,x\gamma)\ge C_3\ell(\gamma)$, where $\ell$ denotes a word metric on $\Pi$.
Moreover, the number of elements in the sphere $S_l$ of radius $l$ in $\Pi$
satisfies $\#S_l\le C_4\,e^{C_5l}$ for some positive constants $C_4,C_5$.
Therefore
\begin{align*}
\sum_{\gamma\in\Pi\setminus\{1\}}
 \exp\Big[-C_2\Big(\frac{d(x,x\gamma)}{t}\Big)^{2}\Big]            
\le&\sum_{\gamma\in\Pi\setminus\{1\}}
 \exp\big[-C'(\ell(\gamma)/t)^{2}\big]                             \\ 
\le&\;\sum_{l=1}^\infty\exp\big[-C'(l/t)^{2}\big]\,C_4\,e^{C_5l} \\
\le&\;C_4\exp[-C't^{-2}]\sum_{l=1}^\infty
 \exp\big[-C'(l^{2}-1)+C_5l\big]                        
\end{align*}
for all $t$ such that $0<t\le1$ for some positive constant $C'$.
Since the infinite sum over $l$ converges, hence
the result.
\end{proof}

It follows from the theorem above that the small time asymptotic expansions of $ \tr( k_\pm(t,\bar x,\bar x))$
are equivalent to the small time asymptotic expansions of $ \tr( \widetilde k_\pm(t,x,x)) \Rank(\E)$.
In particular, we get 

\begin{corollary}\label{cor:indexcoeff}
Let $\alpha_0^\E(x)$ denote the coefficient of $t^0$ in the small time asymptotic expansions of $\tr(k_+(t,x,x))- \tr(k_-(t,x,x))$ 
and $\alpha_0(x)$ denote the coefficient of $t^0$ in the small time asymptotic expansions of $\tr(\widetilde k_+(t,x,x))- \tr( \widetilde k_-(t,x,x))$.
Then 
$$
\alpha_0^\E(x) = \alpha_0(x)  \Rank(\E).
$$

\end{corollary}

\end{document}